\documentclass[11pt,a4paper,twoside,openright]{amsart} % for some reason, openany is necessary to make sure the page numbers are on the correct side
\usepackage{mijnpackages}
\usepackage{kontsevich-package}
\usepackage{fullpage}

\numberwithin{equation}{section} % zorgt ervoor dat equations geteld worden per subsectie
\DeclareMathOperator{\sheafhom}{\mathscr{H}\text{\kern -3pt {\calligra\large om}}\,}

\newcommand{\Bunpar}{\Bun_{2,D}}
\newcommand{\Bunpard}[1]{\Bun_{2,D}^{#1}}
\renewcommand{\Bunrel}[1]{\Bunpar^{\rel, #1}}
\renewcommand{\Bunextd}[1]{\Bunpar^{\mathord{+},#1}} % de iets grotere stack, met een beetje torsie

\renewcommand{\cohpar}{\Coh_{0,D}^1}
\renewcommand{\barcohpar}{\overline{\Coh}_{0,D}^{1}}

\renewcommand{\heckelocalr}[1]{\heckeglobal_{#1}^{(0,1)}}
\renewcommand{\heckelocall}[1]{\heckeglobal_{#1}^{(1,0)}}
\newcommand{\heckelocalz}[1]{\heckeglobal_{#1}^{0}} % the Hecke operator wrt k_x^0

\newcommand{\Extcusp}{\underline{\Ext}} % this is $\mathcal S_I$ in the thesis
\newcommand{\Picpar}{\Pic_D}

\newcommand{\pEtilde}{\tilde{\cE}^\bullet} % the special bundle in $\Bunrel 2$ with $\Gm \times \Gm$ automorphisms
\newcommand{\cEtilde}{\tilde{\cE}} % the special bundle in $\Bunrel 2$ with $\Gm \times \Gm$ automorphisms
\newcommand{\pEhat}{\hat{\cE}^\bullet} % the special bundle in $\Bunrel 2$ with $\Gm$ automorphisms
\newcommand{\cEhat}{\hat{\cE}} % the special bundle in $\Bunrel 2$ with $\Gm$ automorphisms

\newcommand{\tildeLaumon}{\tilde \cL_E}

\newcommand{\closurepEtilde}{\overline{ \{\pEtilde\}}}

\author{Niels uit de Bos}
\title[An explicit geometric Langlands correspondence]{An explicit geometric Langlands correspondence
for the projective line minus four points}
\makeindex
\begin{document}

\maketitle

\begin{abstract}
  This article deals with
  the tamely ramified geometric Langlands correspondence
  for $\GL_2$ on $\bP_{\bF_q}^1$, where $q$ is a prime power,
  with tame ramification at four distinct points $D = \{\infty, 0,1, t\} \subset \bP^1(\bF_q)$.
  We describe in an explicit way
  (1) the action of the Hecke operators
  on a basis of the cusp forms,
  which consists of $q$ elements
  (\cref{thm:formula-hecke-operators});
  and
  (2) the correspondence
  that assigns to a pure irreducible rank 2 local system $E$ on $\bP^1 \setminus D$
  with unipotent monodromy
  its Hecke eigensheaf $\Aut_E$
  on the moduli space $\Bunpar$ of rank 2 parabolic vector bundles
  (\cref{thm:intro-eigensheaf}).
  We define a canonical embedding $\bP^1 \setminus D \incl \Bunpard 1$
  and show with a new proof
  that $\Aut_E|_{\Bunpard 1}$ is the intermediate extension of $E$.
\end{abstract}

% \tableofcontents

\section{Introduction}
This article proves the main results
from the author's thesis about the tamely ramified geometric Langlands correspondence
for $\GL_2$ on $\bP_{\bF_q}^1$, where $q$ is a prime power,
with tame ramification at four distinct points $D = \{\infty, 0,1, t\} \subset \bP^1(\bF_q)$.
We describe in a completely explicit way
(1) the action of the Hecke operators
on a basis of the cusp forms,
which consists of $q$ elements
(\cref{thm:formula-hecke-operators});
and
(2) the correspondence
that assigns to a pure irreducible rank 2 local system $E$ on $\bP^1 \setminus D$
with unipotent monodromy
its Hecke eigensheaf
(\cref{thm:intro-eigensheaf}).
Roughly speaking,
the correspondence says 
that the local system is its own Hecke eigensheaf.

The calculation of the matrix coefficients for the Hecke operators
was the original motivation for the work in this thesis.
Kontsevich
\cite[Section 0.1]{Kontsevich2009}
provides a formula,
but this formula lacks a proof or explanation
and it is not entirely clear what the terms mean.
Moreover,
Mellit, Golyshev and van Straten
noticed that the published formulas of Kontsevich contain misprints, but they were able
to able to guess a correction term that made
the Hecke operators commute.  They
used this for their computer computations of Hecke eigensheaves.
Lastly, the formula also exhibits interesting symmetries that warrant an explanation.
For example,
the formula is symmetric in
the support of the cusp form,
which is a set of rank 2 vector bundles on $\bP^1$ with a parabolic structure at $D$,
and the locus of the Hecke operator,
which is a point in $\bP^1$
---
two seemingly unrelated objects.
The original aim was to prove and provide the correct formulas
and to explain this symmetry.
This led to the following theorem on the action of the Hecke operators.
Here $\{F_z\}_{z \in \bF_q}$ and $\{F_z^0\}_{z \in \bF_q}$ are specific bases of the cusp forms in degree 1 and 0, respectively,
defined in \cref{defn:cusp-forms-fz}.
For $z \in D$,
we denote by $M_z \colon \bP^1 \isom \bP^1$
the unique M\"obius transformation that preserves $D$
and sends $\infty$ to $z$.

\begin{thm}[Theorem \ref{thm:formula-hecke-operators}]
  Let $z \in \bF_q$ and let $x \in \bP^1$.
  Let $\pT \in \cohpar$ be a parabolic torsion sheaf supported at $x$
  with automorphism group $\Gm$ 
  (automatic if $x \not \in D$)
  and let $\heckelocal x$ be the Hecke operator with respect to $\pT$.
  First suppose $x \neq \infty$.
  Then
  \begin{equation*}
    \heckelocal x F_z^0 = \sum_{y \in \bF_q} \alpha_{z,y}^x F_y
  \end{equation*}
  where
  for all $y \in \bF_q \setminus \{x\}$,
  \begin{equation*}
    \begin{split}
      \alpha_{z,y}^x &=
      \# \left\{
        r \in \bF_q^* :
        z = \frac{(yr - x)((y - 1)(y - t)r - (x - 1)(x - t))}{- (x - y)^2 r}
      \right\} 
      \\
      &\phantom{=}-
      \begin{cases}
        0 & \text{if $x \in D$ and $y \in D$} \\
        1 & \text{if $x \in D$ or $y \in D$, but not both} \\
        2 & \text{otherwise}
      \end{cases}
      \\
      &\phantom{=}-
      \begin{cases}
        q & \text{if $z \in D$ and $y = M_z(x)$} \\
        0 & \text{otherwise}
      \end{cases}
    \end{split}
  \end{equation*}
  and
  \begin{equation*}
    \alpha_{z,x}^x
    = \# \left\{
      r \in \bF_q :
      z = -(yr - 1)((y - 1)(y - t)r - (2y - (1+ t)))
    \right\}
    - q + 1.
  \end{equation*}
  If $x = \infty$,
  then the same holds with
  \begin{equation*}
    \alpha_{z,y}^\infty
    =
    \begin{cases}
      -1 & \text{if $z = y$} \\
      0 & \text{otherwise}
    \end{cases}.
  \end{equation*}
\end{thm}

In the course of proving this theorem,
we reached a better understanding of this symmetry,
which we will explain shortly.
This in turn led to a new way to prove
that the Hecke eigensheaf associated to $E$
is the intermediate extension of $E$,
which can canonically be considered as a local system on an open substack $\Bunrel 1$
of the moduli space $\Bunpar$ of rank 2 vector bundles with parabolic structure at $D$.
(\cref{thm:intro-eigensheaf}).
(A proof of the Langlands correspondence for rank 2 local systems with unipotent monodromy appears in
\cite{Drinfeld1987}.)
Here $\Bunpard 1 \subset \Bunpar$ is the substack of parabolic bundles
with underlying bundle of degree 1
and $j \colon \bP^1 \setminus D \incl \bP^1$ denotes the inclusion.

\begin{thm}
  \label{thm:intro-eigensheaf}
  There exists a canonical open embedding
  \begin{equation*}
    j^\rel \colon \cohpar \incl \Bunpar^1
  \end{equation*}
  such that for any pure irreducible rank 2 local system $E$ on $\bP^1 \setminus D$ with unipotent monodromy,
  the Hecke eigensheaf $\Aut_E$ on $\Bunpar$ associated to $E$ satisfies
  \begin{equation*}
    \Aut_E|_{\Bunpar^1} = j^\rel_{!} j_{!*} E.
  \end{equation*}
\end{thm}

This article is based on the author's thesis
\cite{mythesis},
where additional details may be found.

\subsection{Overview of the contents}

\cref{sec:recollections-on-gometric-objects}
provides some recollections on the geometric objects that are central in this article:
parabolic coherent sheaves
(\cref{sec:parabolic-coherent-sheaves}),
their moduli spaces
(\cref{sec:moduli-spaces-of-parabolic-sheaves})
and modifications
(\cref{sec:modifications-of-parabolic-sheaves});
and the 
the Hecke stack
(\cref{sec:geometric-hecke-ops}),
which is defined in terms of parabolic coherent sheaves
and is used to define (also in \cref{sec:geometric-hecke-ops})
the geometric Hecke operators and eigensheaves.

In \cref{sec:cusp-forms-and-relevant-locus},
we define the geometric cusp conditions
and the \emph{relevant locus}
(\cref{def:relevant-locus}),
which is the open substack
on which the cusp forms are supported.
The next section,
\cref{sec:identifying-cohpar-with-bunrel},
shows
that the connected component of
this relevant locus in degree 1,
and therefore in every odd degree,
is canonically isomorphic to $\cohpar$,
the moduli stack of degree 1 torsion sheaves on $\bP^1$
with parabolic structure at $D$.
\cref{sec:cusp-forms-complete-characterisation}
then concludes our analysis of the cusp conditions
by providing a complete characterisation
of the cusp forms.

\Crefrange{sec:calculation-of-all-length-1-lower-modifications}{sec:determining-hecke-operators}
are the core computational sections of this article.
Working on $\bF_q$-points,
we first determine all length 1 lower modifications 
of parabolic bundles in the relevant locus
(\cref{sec:calculation-of-all-length-1-lower-modifications})
and then use this to provide a formula for the action of the Hecke operators
on the cusp forms
(\cref{sec:determining-hecke-operators}).
In \cref{sec:argument-on-function-level},
we show on the level of $\bF_q$-points that the local system is the Hecke eigensheaf:
we prove that the trace-of-Frobenius function of the local system
is in fact the trace-of-Frobenius function of the Hecke eigensheaf
associated to that local system
with a very short and simple calculation.

In the last sections,
\crefrange{sec:cohomological-properties-of-e}{sec:proof-of-hecke-property},
we construct for every pure irreducible rank 2 local system $E$ on $\bP^1 \setminus D$
with unipotent monodromy
the associated Hecke eigensheaf $\Aut_E$.
\Cref{sec:definition-and-perversity-of-hecke-eigensheaf}
gives the construction and proves that it is perverse.
The remaining sections prove that this is indeed the Hecke eigensheaf associated to $E$:
\cref{sec:decomposition-of-hecke-transform-and-compactification}
shows that the Hecke transform of $\Aut_E$ decomposes as a direct sum of shifted perverse sheaves;
in \cref{sec:hecke-transform-is-intermediate-extension},
we prove that this Hecke transform is the intermediate extension of its restriction
to the open substack that lies over $\bP^1 \setminus D$;
and finally, in \cref{sec:proof-of-hecke-property},
we complete the proof by showing that the Hecke transform is symmetric.

\section{Recollections on the geometric objects}
\label{sec:recollections-on-gometric-objects}

In this section,
we introduce the geometric objects
that play a central role in this ramified geometric Langlands correspondence.

\subsection{Parabolic coherent sheaves}
\label{sec:parabolic-coherent-sheaves}
The tame ramification at $D$ is reflected in the additional data of a \defn{parabolic structure}
on the vector bundles in the moduli space on the automorphic side of the correspondence.
By a \defn{parabolic coherent sheaf} (sometimes simply: \defn{parabolic sheaf}), we will in this article
mean the datum
\begin{equation*}
  (\cF^{(i,x)}, \phi_{(i,x)})_{i \in \bZ, x \in D}
\end{equation*}
of coherent sheaves $\cF^{(i,x)}$ on $\bP^1$
and maps $\phi_{(i,x)} \colon \cF^{(i,x)} \to \cF^{(i + 1, x)}$
such that
\begin{enumerate}
\item for all $x, y \in D$, $\cF^{(0,x)} = \cF^{(0,y)}$;
\item
  for all $i \in \bZ$ and all $x \in D$,
  $\cF^{(i + 2, x)} = \cF^{(i,x)}(x)$;
\item
 the composition
 \begin{equation*}
   \cF^{(i, x)} \xrightarrow{\phi_{(i,x)}}
   \cF^{(i + 1, x)} \xrightarrow{\phi_{(i + 1,x)}}
   \cF^{(i + 2, x)}
 \end{equation*}
 is the map induced by the inclusion $\cO \incl \cO(x)$;
 and lastly
\item
  for all $i \in \bZ$ and $x \in D$,
  $\phi_{(i + 2,x)} = \phi_{(i,x)} \otimes \id_{\cO(x)}$.
\end{enumerate}
We refer to $\cF := \cF^{(0,x)}$ for any $x \in D$ as the underlying coherent sheaf.
Each $(i,x) \in \bZ \times D$ is referred to as a \defn{parabolic degree}.
The degree of a parabolic bundle is defined to be the degree of the underlying sheaf.
By a \defn{parabolic vector bundle} (sometimes: \defn{parabolic bundle}),
we mean a parabolic coherent sheaf $\pF$ as above,
such that $\cF^{(i,x)}$ is a vector bundle for all parabolic degrees $(i,x)$.

A map of parabolic sheaves $\pF \to \pG$
is defined as a collection of chain maps
$f_{(\bullet, x)} \colon \cF^{(\bullet, x)} \to \cG^{(\bullet, x)}$ for every $x \in D$,
such that the map $f_{(i + 2, x)}$ is identified with the map $f_{(i, x)} \otimes \id_{\cO(x)}$
for all parabolic degrees $(i,x)$.

It is straightforward to generalize these definitions,
for example to sheaves on other curves with different divisors;
see for example
\cite[section 2.2]{heinloth2004}.
We will only use parabolic sheaves of the type just described,
and families of such sheaves.

A sequence of parabolic sheaves is defined to be \defn{exact}
if it is an exact sequences of sheaves in every parabolic degree.
Many
constructions from homological algebra can be carried over from sheaves
to parabolic sheaves by doing them in every parabolic degree.

As an example of parabolic sheaves,
consider \defn{parabolic line bundles},
which we define as parabolic sheaves $\pL$
such that $\cL^{(i,x)}$ is a line bundle for every parabolic degree $(i,x)$.
It follows directly from the definitions that for every $(i,x)$,
$\cL^{(i,x)}$ is either $\cL( \floor{\frac i 2} x)$ or $\cL( \ceil{\frac i 2} x)$,
and that $\phi_{(i,x)}$ is either the identity or the natural map induced by
$\cO \incl \cO(x)$.
We introduce the following notation for parabolic line bundles:
for $I \subset D$ a subset and $\cL$ a line bundle on $\bP^1$,
we denote by \index{$(L, I)$@$(\cL, I)$}
\begin{equation*}
  (\cL, I)
\end{equation*}
the parabolic line bundle $\pL$
that satisfies $\cL^{(1,x)} = \cL(x)$ for $x \in I$
and $\cL^{(1,x)} = \cL$ for $x \in D \setminus I$.
Every parabolic line bundle is of this form.

\subsection{Moduli spaces of parabolic sheaves}
\label{sec:moduli-spaces-of-parabolic-sheaves}

The main geometric objects in this article are moduli spaces of parabolic sheaves.
We define the moduli stack
\index{$Bun_2$@$\Bunpar$}
\begin{equation*}
  \Bunpar :=
  \stackbuilder{\pE}
  {\text{$\cE$ has rank 2 and for all $i \in \bZ$, $x \in D$} \\
    \deg \cE^{(i,x)} = \deg \cE + i
    }.
\end{equation*}
For $d \in \bZ$,
we denote by $\Bunpard d \subset \Bunpar$ the connected component
\index{$Bun_2d$@$\Bunpard d$}
classifying rank 2 parabolic vector bundles of degree $d$.

Let $\cE$ be a rank 2 vector bundle on $\bP^1$.
A \defn{flag} of $\cE$ at a point $y \in \bP^1$
is a one-dimensional linear subspace $\ell$
of the fiber $\cE|_y$ of $\cE$ at $y$.
A parabolic structure $\pE$ on $\cE$ that lies in $\Bunpar$,
i.e.,
a parabolic sheaf $\pE$ with underlying sheaf $\cE$
satisfying $\deg \cE^{(i,x)} = \deg \cE + i$ for all $i \in \bZ$, $x \in D$,
is the same as the datum of a flag $\ell_x$ on $\cE$ at $x$
for each $x \in D$.
Indeed,
given a parabolic structure,
we can define $\ell_x$ as the image of $\cE^{(-1,x)}|_x$ in $\cE|_x$.
Conversely,
given a flag $\ell_x \subset \cE|_x$,
we can define $\cE^{(-1,x)}$ as the kernel of the map
$\cE \to \cE|_x/\ell_x$
and $\phi_{(-1,x)}$ as the inclusion,
which completely determines the parabolic structure.

Let $\cE$ be a rank 2 vector bundle on $\bP^1$
and for each $x \in D$, let $\ell_x \subset \cE|_x$ be a flag.
We denote by \index{$(E, lx)$@$(\cE, (\ell_x)_{x \in D})$}
\begin{equation*}
  (\cE, (\ell_x)_{x \in D})
\end{equation*}
the parabolic vector bundle $\pE \in \Bunpar$
with underlying vector bundle $\cE$
and parabolic structure given by the flags $(\ell_x)_{x \in D}$
as explained above.

Another crucial moduli stack is the stack
\index{$Coh$@$\cohpar$}
\begin{equation*}
  \cohpar :=
  \stackbuilder{\pT}
  {
    \text{for all $i \in \bZ$, $x \in D$,} \\
    \text{$\cT^{(i,x)}$ has rank 0 and degree 1}
  }.
\end{equation*}
Note that for $\pT \in \cohpar$,
the torsion sheaves $\cT^{(i,x)}$ in all parabolic degrees
are supported at the same point $p \in \bP^1$.
In fact, if $p := \Supp \cF$ does not lie in $D$,
then $\pF$ is the skyscraper sheaf of length 1 supported at $p$
(the parabolic structure is trivial in this case)
and if $p$ does lie in $D$,
then it is isomorphic to one of of the following parabolic torsion sheaves,
where $k_p$ denotes the skyscraper sheaf of length 1 supported at $p$
and the first $k_p$ to be displayed is in parabolic degree zero:
\index{$k_p^0$}
\index{$k_p^{(1,0)}$}
\index{$k_p^{(0,1)}$}
\begin{align*}
  k_p^0 &:= ( \mathellipsis \xrightarrow{0} k_p \xrightarrow{0} k_p \xrightarrow{0} k_p \xrightarrow{0} \mathellipsis ), \\
  k_p^{(1,0)} &:= ( \mathellipsis \xrightarrow{0} k_p \xrightarrow{1} k_p \xrightarrow{0} k_p \xrightarrow{1} \mathellipsis ), \; \text{or} \\
  k_p^{(0,1)} &:= ( \mathellipsis \xrightarrow{1} k_p \xrightarrow{0} k_p \xrightarrow{1} k_p \xrightarrow{0} \mathellipsis ).
\end{align*}
The sheaves $k_p^0$ with $p \in D$ are the only torsion sheaves $\pT \in \cohpar$ have $\Gm \times \Gm$ as their automorphism group;
the others have $\Gm$ as their automorphism group.
The support map
\index{$Supp$@$\Supp$}
\begin{equation*}
  \Supp \colon \cohpar \to \bP^1
\end{equation*}
that sends a parabolic torsion sheaf to its support,
is the universal map to the coarse moduli space.
The preimage of a point $p \in D$ under the support map
is isomorphic to $\BGm \times [\{(x,y) \in \bA^2 : xy = 0\} / \Gm]$,
where $\Gm$ acts anti-diagonally.
In fact, it is not difficult to see that $\cohpar$ is locally around a point $p \in D$
isomorphic to $\BGm \times [\bA^2/\Gm]$.
% the isomorphism being given by
% $(\lambda, \mu) \mapsto (\mathellipsis \xrightarrow{\lambda} k \xrightarrow{\mu} k \xrightarrow \mathellipsis)$.
(See \cite[lemma 3.6]{heinloth2004} for details.)

We define $\barcohpar$ as the rigidification \index{$Cohparbar$@$\barcohpar$}
(in the sense of
\cite{acv}) % https://arxiv.org/pdf/math/0106211.pdf
of $\cohpar$ with respect to the central automorphisms $\Gm$ that scale the sheaf in every parabolic degree
by the same scalar.
This stack has a more concrete description as the stack classifying pairs
$(\pT, s)$ with $\pT \in \cohpar$ and $s \in \rH^0(\bP^1, \cT) \setminus \{0\}$.
The map $\Supp \colon \cohpar \to \bP^1$ factors through
$\Supp \colon \barcohpar \to \bP^1$.
This latter map is an isomorphism over $\bP^1 \setminus D$,
so that we get an inclusion
\begin{equation}
  \label{eq:inclusion-p-setminus-d-into-barcohpar}
  j \colon \bP^1 \setminus D \incl \barcohpar.
\end{equation}

\subsection{Modifications of parabolic sheaves}
\label{sec:modifications-of-parabolic-sheaves}
A \defn{length one lower modification} of a coherent sheaf $\cF$ on $\bP^1$
at a point $x \in \bP^1$
is a subsheaf $\cF' \subset \cF$
such that the cokernel is supported at $x$ and has length 1.
Similarly, a \defn{length one upper modification} of $\cF$
is an inclusion $\cF \subset \cF''$
such that the cokernel satisfies the same properties.

This definition can easily be extended to parabolic sheaves:
a length one lower modification of a parabolic sheaf $\pF$
is a parabolic subsheaf $\pG \subset \pF$
such that in every parabolic degree $(i,x)$,
the inclusion $\cG^{(i,x)} \subset \cF^{(i,x)}$
is a length 1 lower modification of coherent sheaves.
The cokernel $\pT$ of such a modification lies in $\cohpar$.

The modifications $\pE \subset \pF$ of a parabolic bundle $\pF$ at a point $x$ outside of $D$
are classified by the flags in $\cF|_x$:
the flag corresponding to $\pE \subset \pF$
is the image of $\cE|_x$ in $\cF|_x$.
We denote the modification corresponding to $\ell \subset \cE|_x$
by $T_x^\ell \pE$.
Similarly,
modifications $\pE \subset \pF$ of a parabolic bundle $\pF$ at a point $x \in D$
are classified by a pair of flags
$(\ell', \ell) \in \bP^1(\cF^{(-1,x)}|_x) \times \bP^1(\cF|_x)$
such that at least one of the following two conditions hold:
(1) $\ell$ is the image of $\cF^{(-1,x)}|_x$ in $\cF|_x$;
or
(2) $\ell'$ is the image of $\cF^{(-2,x)}|_x$ in $\cF^{(-1,x)}|_x$.
The isomorphism class of the cokernel of the modification is determined by which of these conditions hold;
e.g., the cokernel is isomorphic to $k_x^0$ if and only if both conditions hold.
We denote the modification corresponding to 
$(\ell', \ell)$
by $T_x^\ell \pE$ if only condition (2) holds,\index{$T_x^\ell$}
by $\leftmodif{\ell'}{\pE}$ if only condition (1) holds,\index{$T$@$\leftmodif{\ell'}{}$}
and by $T_x \pE$\index{$T_x$} if both conditions hold.

\subsection{Geometric Hecke operators and Hecke eigensheaves}
\label{sec:geometric-hecke-ops}
The geometric Hecke operators are defined in terms of correspondences
involving the \defn{Hecke stack}\index{$H$@$\cH$|see{Hecke stack}} $\cH$ of length 1,
which is defined as
\begin{equation*}
  \cH :=
  \stackbuilder{0 \to \pF \to \pE \to \pT \to 0 \text{ exact}}
  {\pF, \pE \in \Bunpar, \; \pT \in \cohpar}.
\end{equation*}
The Hecke correspondence is the diagram
\begin{equation}
  \label{eq:hecke-diagram-plain}
  \begin{tikzcd}
   & \cH \ar[ld, "p", swap] \ar[rd, "q"] & \\
    \Bunpar \times \barcohpar && \Bunpar
  \end{tikzcd}
\end{equation}
where the maps $p,q$
are defined by
\begin{equation}
  \label{eq:defn-p-q-from-hecke-stack}
  \begin{split}
    p(0 \to \cE' \to \cE \to \cT \to 0) &= (\cE, [\cT]) \\
    q(0 \to \cE' \to \cE \to \cT \to 0) &= \cE'
  \end{split}.
\end{equation}
The \defn{global Hecke operator}
is the map $\heckeglobal$
on the bounded derived categories of $\ell$-adic sheaves on $\Bunpar$ and $\Bunpar \times \cohpar$
defined as
\begin{equation*}
  \heckeglobal \colon \boundedderivcat(\Bunpar, \bQ_\ell) \to \boundedderivcat(\Bunpar \times \barcohpar, \bQ_\ell),
  \qquad
  F \mapsto p_! q^* F.
\end{equation*}
A \defn{Hecke eigensheaf} for an irreducible local system $E$ on $\bP^1 \setminus D$
is a perverse sheaf $F$ on $\Bunpar$ such that there exists an isomorphism
\begin{equation*}
  \heckeglobal F \cong F \boxtimes j_{!*} E
\end{equation*}
where $j \colon \bP^1 \setminus D \incl \barcohpar$ is the inclusion
defined in \eqref{eq:inclusion-p-setminus-d-into-barcohpar}.

The maps $p$ and $q$ are both smooth.
The map $q$ is proper,
while $p$ is proper only over $\Bunpar \times (\barcohpar \setminus \Supp^{-1}(D))$.
(We will later compactify the map $p$.)
See \cite[remark 6.3]{heinloth2004}
or  \cite[lemma 4.12]{mythesis}
for proofs of these facts.
(Note that these proofs use slightly different definitions of $p$ and $q$,
but duality of vector bundles carries over that definition into the definition given here.)

The \defn{local Hecke operators}
can be defined as restrictions of the global Hecke operators:
for $\pT \in \barcohpar(\bF_q)$,
the local Hecke operator
\begin{equation*}
  \heckelocal{\pT} \colon \boundedderivcat(\Bunpar, \bQ_\ell) \to \boundedderivcat(\Bunpar, \bQ_\ell),
\end{equation*}
is defined as $(\id \times \pT)^* \circ \heckeglobal$.
For $x \in D \subset \bP^1(\bF_q)$,
we write
\begin{equation*}
  \heckelocal x := \heckelocal{k_x^{0}},
  \qquad
  \heckelocalr x := \heckelocal{k_x^{ (0,1)}},
  \quad \text{and} \quad
  \heckelocall x := \heckelocal{k_x^{(1,0)}}
\end{equation*}
where the $k_x^\bullet$ represent the three isomorphism classes of degree 1 parabolic torsion sheaves
supported on $x$ (see \cref{sec:moduli-spaces-of-parabolic-sheaves}).
If $\pT \in \barcohpar(k)$ is supported at $x \in \bP^1(x) \setminus D$,
then we also write $\heckelocal{x} := \heckelocal{\pT}$.

Lastly, we define \defn{elementary Hecke operators}.
For every $x \in D$,
the elementary Hecke operator $T_x$ at $x$\index{$T_x$}
is the isomorphism
\index{$T_x$}
\begin{equation*}
  T_x \colon \Bunpar \isom \Bunpar
\end{equation*}
obtained by shifting the parabolic structure of $\pE \in \Bunpar$ at $x$:
we map
$\pE$ to the parabolic bundle $T_x \pE$ with underlying sheaf $\cE^{(-1,x)}$
and parabolic structure induced by the parabolic structure of $\pE$.
(We obtain chains $(T_x \pE)^{(\bullet, y)}$ for $y \in D \setminus \{x\}$
by noting that the $\phi_{(i,x)}$ in the definition of parabolic coherent sheaves
are isomorphisms outside of $x$,
and then gluing.)
It follows directly from the definitions that for all $x \in D$,
the local Hecke operator
$\heckelocalz x$ is the same as pulling back by $T_x$.
Elementary Hecke operators at different points in $D$ commute
and for a subset $I \subset D$,
we denote by $T_I$ the composition of the elementary Hecke operators $T_x$ with $x \in I$.
We refer to the sheaf  $(T_I^{-i} \pE)^0$ as the underlying sheaf in parabolic degree \index{parabolic degree} $(i, D)$.

Let $x \in D$ and let $\pE$ be a parabolic vector bundle.
Every length one lower modification of $\pE$ with cokernel $k_x^0$
is of the form $T_x \pE \incl \pE$.
The definition of $T_x$ therefore does not conflict with the notation introduced in the previous section.

\section{Cusp forms and the relevant locus}
\label{sec:cusp-forms-and-relevant-locus}

\subsection{Definition of the cusp condition}
\label{sec:definition-of-cusp-condition}
% \innote{The setup here is equivalent to, but different from
%   the treatment in my thesis:
%   I treat all $I \subset D$ at the same time.}
Classically, the cusp condition on an automorphic form is defined in terms of
the vanishing of a certain integral.
The geometric analogue of this condition
is defined in terms of a correspondence
\begin{equation}
  \begin{tikzcd}
    & \Extcusp
    \ar[ld, "s", swap]
    \ar[rd, "t"] & \\
    \Bunpar & &
    \Picpar \times \Picpar
  \end{tikzcd}
\end{equation}
where $\Picpar$ denotes the moduli stack that classifies parabolic line bundles
and $\Extcusp$ denotes the moduli stack that classifies short exact sequences
\begin{equation}
  \label{eq:exension-of-parabolic-line-bundles}
  0 \to \pL \to \pE \to \pM \to 0
\end{equation}
with $\pL, \pM \in \Picpar$ and $\pE \in \Bunpar$.
The map $s$ sends an extension
as in \cref{eq:exension-of-parabolic-line-bundles} to $\pE$,
while $t$ sends it to $(\pL, \pM)$.

Note that $\Pic_D = \sqcup_{I \subset D} \Pic$,
where $\Pic$ denotes the stack of (non-parabolic) line bundles on $\bP^1$,
since every parabolic line bundle is of the form $(\cL, I)$ for some $\cL \in \Pic$ and $I \subset D$.
Note also, that for an extension as in \cref{eq:exension-of-parabolic-line-bundles},
the condition on the degrees of $\pE \in \Bunpar$
(namely,
$\deg \cE^{(i,x)} = \deg \cE + i$)
implies that there exists $I \subset D$
such that $\pL = (\cL, I)$ and $\pM = (\cM, D \setminus I)$.

We say that a perverse sheaf $F$ on $\Bunpar$ satisfies the \defn{cusp condition},
if
\begin{equation}
  \label{eq-defn:cusp-condition-on-sheaves}
  \bR t_! s^* F = 0.
\end{equation}

Let $f \colon \Bunpar(\bF_q) \to \bQ_\ell$ be the trace-of-Frobenius function
(see e.g. \cite[(1.1)]{Laumon1987})
associated to a perverse sheaf $F$ on $\Bunpar$.
If $F$ satisfies the cusp condition,
then for all line bundles $\cL, \cM$ on $\bP_{\bF_q}^1$ and all $I \subset D$,
$f$ satisfies the equality
\begin{equation}
  \label{eq:recall-cups-condition-fq-pts}
  \sum_{\substack{(\pL \incl \pE \surj \pM)\\ \in \Ext^1(\pM, \pL)(\bF_q)}} f(\pE) = 0
\end{equation}
where $\pL = (\cL, I)$ and $\pM = (\cM, D \setminus I)$.

\subsection{The relevant locus}
\label{sec:relevant-locus}
We now determine the subspace of $\Bunpar$
on which the cusp forms are supported,
which we call the relevant locus.

\begin{defin}
  \label{def:relevant-locus}
  We define the \defn{relevant locus} of $\Bunpard 1$
  as the open substack
  \index{$Bunreld$@$\Bunrel{d}$}
  \begin{equation*}
    \Bunrel 1 \subset \Bunpard 1
  \end{equation*}
  that classifies parabolic bundles $\pE \cong (\cO(1) \oplus \cO, (\ell_x)_{x \in D})$
  satisfying
  \begin{enumerate}
  \item  $\ell_x = \cO(1)|_x$ for at most one $x \in D$; and
  \item
    the flags do not come from a global section $\sigma \colon \cO \to \cO(1)$,
    i.e.,
    there is no $\sigma \colon \cO \to \cO(1)$ such that for all $x \in D$, $\ell_x = (\sigma|_x : 1) \in \bP^1(\cO(1)|_x \oplus \cO|_x)$.
  \end{enumerate}
  Note that bundles of the form $(\cO(1) \oplus \cO, (\ell_x)_{x \in D})$ that do not satisfy the second condition
  are isomorphic to $(\cO(1), \emptyset) \oplus (\cO, D)$.

  The relevant locus $\Bunrel d$ of $\Bunpar^d$
  is defined as the image of $\Bunrel 1$ under the isomorphism
  \begin{equation*}
    (T_\infty)^{1 - d} \colon \Bunpar^1 \isom \Bunpar^d
  \end{equation*}
  where $T_\infty$ is the elementary Hecke operator at $\infty$ (see \cref{sec:geometric-hecke-ops}).
\end{defin}

\begin{thm}
  \label{prop:cusp-forms-vanish-outside-relevant-locus}
  %% A function $F$ on $\Bunpar$ satisfies the cusp condition
  %% if and only if it satisfies all of the following conditions:\note{Moet alle gelijkheden nou isoms zijn?}
  Let $d \in \bZ$, $n \in \bZ_{\geq 1}$ and let $f \colon \Bunpar^d(\bF_{q^n}) \to \bQ_\ell$ be a function that
  satisfies the cusp conditions.
  Then $f$ vanishes on $\Bunpar^d(\bF_{q^n}) \setminus \Bunrel d(\bF_{q^n})$.
\end{thm}
The above theorem determines the open substack $\Bunrel d \subset \Bunpar^d$ uniquely.
We will also see that $\Bunrel d = T_x^{1 - d} \Bunrel 1$
for any $x \in D$,
so in the definition of the relevant locus in degree $d \neq 1$,
we could have taken $T_x$ with any $x \in D$ instead of $T_\infty$
\begin{proof}[Proof of the theorem]
  Since the elementary Hecke operators $T_x$ with $x \in D$
  are defined as a shift in the parabolic structure,
  they can naturally be extended to $\Extcusp$ and $\Picpar$
  and these extensions commute with the maps $s$ and $t$.
  It follows that
  % \begin{prop}
  %   \label{prop:hecke-operators-preserve-cusp-forms}
  %   Hecke operators preserve cusp forms.
  % \end{prop}
  % This can be proven using
  % \cite[Lemma 9.8]{FGVConjecture}
  % and the methods used there.
  % We use this result in 
  % \cref{sec:cusp-conditions-and-relevant-locus},
  % but it could in principle be avoided
  % by slightly adapting and extending the proofs in that section.
  $f$ is a cusp form if and only if $f \circ T_x$ is.
  It therefore suffices to prove that any cusp form vanishes outside of $\Bunrel 1$.

  Let $\cL, \cM$ be line bundles on $\bP^1$ and let $I \subset D$.
  Then the natural map
  \begin{equation*}
    \Ext^1( (\cM, I), (\cL, D \setminus I)) \to \Ext^1( \cM(I), \cL)
  \end{equation*}
  that sends an extension of parabolic sheaves
  to its associated extension in parabolic degree $(1, I)$
  (i.e., obtained by applying $T_I^{-1}$ and taking the degree zero underlying sheaf),
  is an isomorphism.
  (The inverse is given by pulling back the extension in $\Ext^1(\cM(I), \cL)$
  along the maps that define the parabolic structure of $T_I^{-1}(\cM, I) = (\cM(I), \emptyset)$,
  with which the maps in all parabolic degrees can be recovered.)

  The fiber of the map
  $\Extcusp \to \Picpar \times \Picpar$ over the point $((\cL, D \setminus I), (\cM, I))$
  is $\Ext^1( (\cM, I), (\cL, D \setminus I)) \cong \Ext^1(\cM(I), \cL)$,
  which contains only the trivial extension
  if and only if
  \begin{equation}
    \label{eq:condition-to-be-semisimple}
    \deg \cL - \deg \cM - \# I \geq -1.
  \end{equation}
  Therefore,
  any $f \colon \Bunpard 1(\bF_q) \to \bQ_\ell$
  satisfying the cusp condition
  vanishes on parabolic bundles of the form
  \begin{equation*}
    (\cL, D \setminus I) \oplus (\cM, I)
  \end{equation*}
  where $\cL, \cM$ and $I$ satisfy \cref{eq:condition-to-be-semisimple}.
  This includes all parabolic bundles $\pE$ with underlying bundle of the form
  $\cE = \cO(1 + m) \oplus \cO(-m)$ for some $m \in \bZ_{> 0}$:
  it has a subbundle of the form $(\cL = \cO(1 + m), D \setminus I)$ for some $I \subset D$;
  the quotient is of the form $(\cM, I)$
  and the inequality \ref{eq:condition-to-be-semisimple} then automatically holds,
  so that $\pE$ is indeed isomorphic to such a $(\cL, D \setminus I) \oplus (\cM, I)$.
  By the same reasoning
  (i.e., taking $(\cL, D \setminus I)$ the maximal destabilizing subbundle),
  the same holds for $\pE = (\cO(1) \oplus \cO, (\ell_x)_{x \in D})$ such that
  $\ell_x = \cO(1)|_x$ for at least 2 $x \in D$
  (which implies $\#I \leq 2$, and therefore the inequality \cref{eq:condition-to-be-semisimple} holds).
  Lastly, it also includes the bundle $(\cO, D) \oplus (\cO(1), \emptyset)$
  (take $I = \emptyset$, $\cL = \cO$ and $\cM = \cO(1)$).
\end{proof}

There are precisely two parabolic bundles $(\cO(1) \oplus \cO, (\ell_x)_{x \in D}) \in \Bunrel 1$
with $\ell_\infty = \cO(1)|_\infty$.
Hence, we find that the relevant locus $\Bunrel 2$ in degree 2,
which we obtain by applying $T_\infty^{-1}$ to the parabolic vector bundles in $\Bunrel 1$,
contains precisely two parabolic bundles with underlying vector bundle
$\cO(2) \oplus \cO$,
both of which play a very important role in the rest of this article.
One of these is
\begin{equation*}
  \pEtilde := (\cO(2), \emptyset) \oplus (\cO, D).
\end{equation*}
This direct sum decomposition is canonical,
because there are no parabolic maps $(\cO, D) \to (\cO(2), \emptyset)$,
and the automorphism group $\Aut(\pEtilde)$ is therefore canonically isomorphic to
the product
$\Aut((\cO(2), \emptyset)) \times \Aut((\cO,D)) = \Gm \times \Gm$.
We denote the other parabolic bundle in $\Bunrel 2$ with underlying vector bundle $\cO(2) \oplus \cO$
by $\pEhat$.
Its automorphism group consists of the central automorphism $\Gm$.
Any bundle $(\cO(2) \oplus \cO, (\ell_p)_{p \in D})$
such that $\ell_p \neq \cO(2)|_p$ holds for all $p \in D$
is isomorphic to either $\pEtilde$
(if the flags come from a global section)
or $\pEhat$ (otherwise). 

Let $\pT \in \cohpar$ with support $x \in \bP^1$.
Recall (\cref{sec:modifications-of-parabolic-sheaves})
that the length 1 lower modifications of $\pEtilde$ with cokernel $\pT$
are classified
by flags $\ell \in \bP^1(\cEtilde)|_x$
(or, if $x \in D$,
a certain subset of
pairs of flags
$(\ell', \ell) \in \bP^1(\cE^{(-1,x)}|_x, \cE^{0}|_x)$).
The automorphism group $\Gm \times \Gm$ acts on these flags
and flags in the same orbit
define subbundles that are isomorphic as parabolic bundles
(i.e., after forgetting the inclusion).
For each $\pT \in \cohpar$,
there is exactly one open orbit of modifications with cokernel $\pT$
--- we call this orbit the \defn{generic orbit}.
We call the subbundle of the modification associated to this generic orbit
the subbundle of the \defn{generic modification}.
\begin{lemma}
  \label{lem:generic-modifications-of-petilde-lie-in-bunrel}
  Every generic length 1 lower modification of $\pEtilde$ lies in $\Bunrel 1$.
\end{lemma}
\begin{proof}
  Let $(\cE, (\ell_x)_{x \in D}) \overset{i}{\incl} \pEtilde$ be a generic length 1 lower modification of $\pEtilde$.
  Then the underlying vector bundle $\cE$ is isomorphic to $\cO(1) \oplus \cO$,
  because the only flags that define a lower modification with a different underlying vector bundle (which is $\cO(2) \oplus \cO(-1)$),
  are the flags that come from $\cO(2) \subset \cEtilde^0$.
  For each $x \in D$,
  the flag $\ell_x$ can only be $\cO(1)|_x$ if the cokernel of $i$ is supported at $x$;
  in particular,
  there is at most one $x$ with $\ell_x = \cO(1)|_x$.
  Lastly, we can prove that $(\cE, (\ell_x)_{x \in D})$ is not isomorphic to $(\cO(1), \emptyset) \oplus (\cO, D)$,
  by showing that any inclusion
  $(\cO(1), \emptyset) \oplus (\cO, D) \incl \pEtilde$
  corresponds to a non-open orbit.
  (For example, use that there are no parabolic maps $(\cO, D) \to (\cO(2), \emptyset)$.)
\end{proof}

\begin{prop}
  \label{prop:relevant-locus-fq-points-explicit-description}
  The map
  \begin{equation*}
    \alpha \colon \cohpar(\bF_q) \to \Bunrel 1(\bF_q)
  \end{equation*}
  that sends a parabolic torsion sheaf $\pT$
  to the generic modification of $\pEtilde$ with respect to $\pT$,
  is a bijection.
\end{prop}
\begin{proof}
  We prove this by constructing the inverse.
  We show that for every $\pE \in \Bunrel 1$,
  there exists exactly one inclusion
  \begin{equation}
    \label{eq:relevant-locus-fq-points-explicit-description-inclusion}
    i \colon \pE \incl \pEtilde
  \end{equation}
  up to scalar multiplication,
  and that this inclusion corresponds to the generic orbit.
  The inverse of $\alpha$ maps $\pE$ to the cokernel of this inclusion.
  
  We will assume that $\pE = (\cE, (\ell_x)_{x \in D})$
  with $\ell_x \neq \cO(1)|_x$ for all $x \in D$;
  the remaining 8 parabolic bundles in $\Bunrel 1(\bF_q)$
  (for which $\ell_x = \cO(1)|_x$ holds for exactly one $x \in D$)
  can be checked by hand or proven using a small variation of the argument below.

  Any inclusion as in \eqref{eq:relevant-locus-fq-points-explicit-description-inclusion}
  is a scalar multiple of the morphism of parabolic vector bundles
  induced by a morphism of underlying vector bundles
  \begin{equation*}
    \begin{pmatrix}
      \sigma & \tau \\ 0 & 1
    \end{pmatrix}
    \colon \cO(1) \oplus \cO \incl \cO(2) \oplus \cO
  \end{equation*}
  where $\sigma \colon \cO(1) \to \cO(2)$ is non-zero
  (necessary and sufficient for the injectivity of the map)
  and $\tau \colon \cO \to \cO(2)$ is arbitrary.
  For a map of underlying vector bundles to induce a map on parabolic bundles,
  it is necessary and sufficient that
  every flag of the source is mapped into the flag on the target.
  In this case,
  that means that $\ell_x$ should be mapped to $\cO|_x$ for every $x \in D$;
  in other words,
  $(\sigma, \tau)$ is in the kernel of the linear map
  \begin{equation*}
    \rH^0(\bP^1, \cO(1) \oplus \cO(2)) \to (\cO(2)|_x)_{x \in D}, \qquad
    (\sigma', \tau') \mapsto \left( (\sigma' \; \tau') \ell_x \right)_{x \in D}.
  \end{equation*}
  The source $\rH^0(\bP^1, \cO(1) \oplus \cO(2))$ has dimension 5,
  while the target has dimension 4,
  so that the kernel has dimension at least 1.
  The condition that the flags $(\ell_x)_{x \in D}$ do not come from a global section
  (\cref{def:relevant-locus}, condition 2)
  shows that the map is surjective,
  so that the kernel has dimension one.
  A non-zero pair $(\sigma', \tau')$ in the kernel satisfies $\sigma' \neq 0$
  (this follows from the assumptions that $\ell_x \neq \cO(1)|_x$ for all $x \in D$
  and that the flags do not come from a global section),
  so for these $\pE$, we do indeed have a unique inclusion $i \colon \pE \incl \pEtilde$
  up to scalar multiplication.
\end{proof}

\section{Identifying $\cohpar$ with $\Bunrel 1$}
\label{sec:identifying-cohpar-with-bunrel}
The previous section established a bijection on $\bF_q$-points
$\cohpar(\bF_q) \to \Bunrel 1(\bF_q)$.
In this section,
we use that fact to show that there is in fact an isomorphism
\begin{equation*}
  \alpha \colon \cohpar \isom \Bunrel 1
\end{equation*}
of stacks.

We define a substack $\cHcErel$ of the Hecke stack $\cH$ by
\begin{equation*}
  \cHcErel :=
  \stackbuilder{0 \to \pE \to \pEtilde \to \pT \to 0 \; \text{exact}}
  {\text{$\pEtilde \surj \pT$ does not factor through} \\
    \text{either of the direct summands} \\
    (\cO(2), \emptyset) \text{ or } (\cO, D) \\
    \text{of $\pEtilde = (\cO(2), \emptyset) \oplus (\cO, D)$}
  }.
\end{equation*}
This stack classifies exactly the generic length 1 lower modifications of $\pEtilde$
(generic in the sense defined just before \cref{lem:generic-modifications-of-petilde-lie-in-bunrel}).

\begin{lemma}
  \label{lem:chcerel-to-bunrel1-isom}
  The restriction of $q \colon \cH \to \Bunpar^1$ to $\cHcErel$ induces an isomorphism
  \begin{equation*}
    q^\rel \colon \cHcErel \isom \Bunrel 1, \qquad (\pEtilde \overset{\phi}{\surj} \cT) \mapsto (\ker \phi).
  \end{equation*}
\end{lemma}
\begin{proof}
  The map $q^\rel$ is representable
  (it is injective on automorphism groups).
  It is also smooth.
  Indeed, the map on tangent spaces at the point $(\pE \incl \pEtilde \surj \pT) \in \cHcErel$
  is
  \begin{equation}
    \label{eq:map-qrel-on-tangent-spaces}
    \Hom(\pE, \pT) \to \Ext^1(\pE, \pE).
  \end{equation}
  This map appears in the long exact sequence obtained by applying
  $\RHom(\pE, \dash)$
  to the short exact sequence $0 \to \pE \to \pEtilde \to \pT \to 0$.
  The next term in this long exact sequence is
  $\Ext^1(\pE, \pEtilde)$.
  Because $\pE$ and $\pEtilde$ lie in the relevant locus and $\pE$ has one degree lower than $\pEtilde$,
  we can determine the degrees of the direct summands of the underlying vector bundles in every parabolic degree $(i,x)$
  (in even degree $2i$, it is $\cO(i) \oplus \cO(i)$ or $\cO(i + 1) \oplus \cO(i - 1)$;
  in odd degree $2i + 1$, it is $\cO(i + 1) \oplus \cO(i)$)
  and can use them to conclude that
  for every parabolic degree $(i,x)$,
  the group $\Ext^1(\cE^{(i,x)}, \cEtilde^{(i,x)})$ vanishes.
  It follows that 
  \cref{eq:map-qrel-on-tangent-spaces}
  is surjective
  and $q^\rel$ is smooth.

  We already know the map $q^\rel$ is also an isomorphism on $K$-points
  for any field extension $K$ of $\bF_q$
  (\cref{prop:relevant-locus-fq-points-explicit-description}).
  Together with the representability and smoothness,
  this proves it is an isomorphism.
\end{proof}

The automorphisms $\Gm = \Aut( (\cO(2), \emptyset))$
of the parabolic direct summand $(\cO(2), \emptyset) \subset \pEtilde$
form a subgroup of the automorphisms of $\Aut(\pEtilde)$.
By $\BAut(\pEtilde)/\Gm$,
we denote the classifying stack of the quotient $\Aut(\pEtilde)/\Aut((\cO(2), \emptyset))$.

\begin{lemma}
  \label{lem:prel-is-isomorphism}
  The restriction of $p \colon \cH \to \Bunpar \times \barcohpar$ to $\cHcErel$
  induces an isomorphism
  \begin{equation*}
    \begin{gathered}
      p^\rel \colon \cHcErel \xrightarrow{p} \BAut(\pEtilde) \times \barcohpar
      \to \BAut(\pEtilde)/\Gm \times \barcohpar, \\
      (\pEtilde \surj \pT) \mapsto (\pEtilde, [\pT])
    \end{gathered}
  \end{equation*}
\end{lemma}
\begin{proof}
  We prove this by constructing a map
  $\phi \colon \cohpar \to \cHcErel$
  that descends along the cover
  $\cohpar \to \BAut(\pEtilde)/\Gm \times \barcohpar$
  to an inverse of $p^\rel$.
  For simplicity, we define this map on $K$-points for any extension of the field $\bF_q$.
  (For more details on the construction on $S$-points for general $\bF_q$-schemes $S$,
  see \cite[Lemma 7.4]{mythesis}.)
  The main idea behind this isomorphism is that for every $\pT \in \cohpar$,
  there is exactly one generic orbit in the modifications of $\pEtilde$.

  Let $\pT \in \cohpar(K)$
  and let $x \in \bP^1(K)$ denote the support of $\pT$.
  For a parabolic coherent sheaf $\pF$ and $i \in \bZ$,
  we define the sheaf $\cF^{(-i,D)}$
  as $(T_D^i \pF)^0$,
  and we refer to this as the sheaf in parabolic degree $(-i, D)$.
  We write
  $\alpha \colon \cT^{(-1,D)} \to \cT^{(0,D)}$
  and
  $\beta \colon \cT^{(-2,D)} = \cT^{(0, D)} \otimes \cO(-D)|_x \to \cT^{(0,D)}$
  for the maps induced by the parabolic structure on $\pT$.

  A map from $\pEtilde$ to $\pT$
  factors through a map $\pEtilde \to \pEtilde|_x$,
  which in turn defines a chain map
  $(\cEtilde^{(\bullet, D)})|_x \to \cT^{(\bullet, D)}$.
  Knowing the maps
  $(\cEtilde^{(i, D)})|_x \to \cT^{(i, D)}$
  for $i = -2,-1,0$ is sufficient to reconstruct
  the original map $\pEtilde \surj \pT$.
  We leave it to the reader to check that the maps
  $(\cEtilde^{(i, D)})|_x \surj \cT^{(i, D)}$ for $i = -2,-1,0$ in the following diagram
  ($i = 0$ is the top row, $i = -2$ is the bottom row)
  \begin{minipage}{\textwidth} % de minipage en de twee hidewidths zijn om het goed te centreren
  \begin{equation*}
    \hidewidth
    \begin{tikzcd}
      \left(
        \cO(2)|_x \otimes (\cT^{(0,D)} \otimes \cO(-2)|_x)
      \right)
      &[-25pt] \oplus
      &[-25pt]
      \cO|_x \otimes \cT^{(-1,D)}
      \arrow[surj]{r}{(1,\alpha)}
      &
      \cT^{(0,D)}
      \\
      \left(
        \cO(2 - D)|_x \otimes (\cT^{(0,D)} \otimes \cO(-2)|_x)
      \right)
      & \oplus \ar[u]
      &
      \cO|_x \otimes \cT^{(-1,D)}
      \arrow[surj]{r}{(\beta,1)}
      &
      \cT^{(-1,D)}
      \ar[u]
      \\
      \left(
        \cO(2 - D)|_x \otimes (\cT^{(0,D)} \otimes \cO(-2)|_x)
      \right)
      & \oplus \ar[u]
      &
      \cO(-D)|_x \otimes \cT^{(-1,D)}
      \arrow[surj]{r}{(1,\alpha)}
      &
      \cT^{(0,D)} \otimes \cO(-D)|_x
      \ar[u]
    \end{tikzcd}
    \hidewidth
  \end{equation*}
\end{minipage}
  come from a well-defined map $\pEtilde \surj \pT$,
  and that the map $\phi$ defined by sending $[\pT]$ to this map
  descends to an inverse of $p^\rel$.
\end{proof}

Because the decomposition $\pEtilde \cong (\cO(2), \emptyset) \oplus (\cO, D)$
is canonical,
the automorphism group of $\pEtilde$ is canonically $\Gm \times \Gm$.
Since $\cohpar$ is canonically isomorphic to $\barcohpar \times \BGm$,
we have a canonical isomorphism
$\cohpar \isom \barcohpar \times \BAut(\pEtilde)/\Gm$.
\begin{defin}
  \label{defn:alpha}
  We define the isomorphism
  \begin{equation*}
    \alpha \colon \cohpar \isom \Bunrel 1
  \end{equation*}
  as the composition
  \begin{equation*}
    \cohpar \isom \BAut(\pEtilde)/\Gm \times \barcohpar \xrightarrow{(p^{\rel})^{-1}}
    \cHcErel \xrightarrow{q^{\rel}} \Bunrel 1.
  \end{equation*}
\end{defin}
This isomorphism allows us to use the maps
$T_\infty^{1 - d} \circ \alpha \colon \cohpar \isom \Bunrel d$ as charts of $\Bunrel d$.
We also use it to define the following map from $\Bunrel d$ to its coarse moduli space.

\begin{defin}
  \label{defn:pi-d-map-to-coarse-moduli-space-from-relevant-locus}
  Let $d \in \bZ$.
  We define
  \begin{equation*}
    \pi_d := \Supp {} \circ
    \left(
      T_\infty^{1 - d} \circ \alpha
    \right)^{-1}
    \colon
    \Bunrel d \isom \cohpar \to \bP^1.
  \end{equation*}
\end{defin}
For even $d$,
this map depends on the choice of $\infty \in D$:
using $T_x^{d - 1}$ with $x \in D \setminus \{\infty\}$
instead of $T_\infty^{d - 1}$ would result in a different map.

\section{Cusp conditions}
\label{sec:cusp-forms-complete-characterisation}

In this section,
we prove the following complete characteristation of the cusp forms,
which uses the map
$\alpha \colon \cohpar(\bF_q) \to \Bunrel 1(\bF_q)$
is the map from
\cref{prop:relevant-locus-fq-points-explicit-description}.
\begin{thm}
  \label{thm:cusp-condition-fq-pts-necessary-sufficient}
  A function $f \colon \Bunpar^1(\bF_q) \to \bQ_\ell$ satisfies the cusp conditions if and only if
  (1) it vanishes outside of $\Bunrel 1(\bF_q)$;
  and (2)
  \begin{equation}
    \label{eq:cusp-condition-fq-vanish-on-certain-classes}
    \sum_{\pE \in \cP} \frac{f(\pE)}{\# \Aut(\pE)} = 0
  \end{equation}
  for each $\cP$ equal to one of the following sets:
  % for all sets $\cP$ among the following sets?
  \begin{enumerate}
  \item[(2.1)]
    \label{item:characterisation-cusp-forms-vanish-on-a1s-first}
    for each $y \in D$,
    the set
    \begin{equation*}
      \cP_y^{(1,0)} := \im\left(\{k_y^0, k_y^{(1,0)}\} \xrightarrow{\alpha} \Bunrel 1(\bF_q)\right);
    \end{equation*}
  \item[(2.2)]
    \label{item:characterisation-cusp-forms-vanish-on-a1s-second}
    for each $y \in D$,
    the set
    \begin{equation*}
      \cP_y^{(0,1)} := \im\left(\{k_y^0, k_y^{(0,1)}\} \xrightarrow{\alpha} \Bunrel 1(\bF_q)\right);
    \end{equation*}
    and
  \item[(2.3)]
    \label{item:characterisation-cusp-forms-vanish-on-p1s}
    for each section $\sigma \colon \bP^1  \to \cohpar$ of the support map
    $\Supp \colon \cohpar \to \bP^1$,
    the set
    % this is the image of  % $\bP^1 \times \BGm \incl \cohpar \xrightarrow{\alpha} \Bunrel 1$
    \begin{equation*}
      \cP_{\sigma} := \im \left(\bP^1(\bF_q) \xrightarrow{\sigma} \cohpar(\bF_q) \xrightarrow{\alpha} \Bunrel 1(\bF_q)  \right).
    \end{equation*}
  \end{enumerate}
\end{thm}

\begin{remark}
  \label{rmk:cusp-conditions-vanish-on-every-p1-you-choose}
  Let $f \colon \Bunrel 1(\bF_q) \to \bQ_\ell$ be a function that satisfies
  \eqref{eq:cusp-condition-fq-vanish-on-certain-classes}
  for all $\cP$ of the form $\cP_x^\bullet$ with $x \in D$
  (i.e., one of the $\cP$ defined in parts (2.1) and (2.2)).
  Let $\sigma, \sigma' \colon \bP^1 \to \cohpar$ be two sections of
  $\Supp \colon \cohpar \to \bP^1$.
  Then $f$ satisfies
  \eqref{eq:cusp-condition-fq-vanish-on-certain-classes}
  for $\cP = \cP_{\sigma}$ if and only
  it satisfies
  \eqref{eq:cusp-condition-fq-vanish-on-certain-classes}
  for $\cP = \cP_{\sigma'}$.
\end{remark}

\begin{remark}
  For the proof of this theorem,
  it is convenient to first reorder the summation in the cusp condition
  (\cref{eq:recall-cups-condition-fq-pts})
  to sum over the elements of $\Bunpar(\bF_q)$
  instead of over the extensions.
  This then gives the condition that for all line bundles $\cL$ on $\bP^1$
  and all $I \subset D$,
  \begin{equation}
    \label{eq:cusp-condition-summation-reordered}
    \sum_{\pE \in \Bunrel 1(\bF_q)}
    \frac{q - 1}{\#\Aut(\pE)} \# \homparinjsat( (\cL, I), \pE) \cdot f(\pE) = 0
  \end{equation}
  where
  $\homparinjsat$ denotes the set of all parabolic injective morphisms that are saturated
  in every parabolic degree.
\end{remark}

\begin{proof}[Proof of the necessity of the conditions in \cref{thm:cusp-condition-fq-pts-necessary-sufficient}]
  The necessity of condition (1)
  (vanishing outside of the relevant locus)
  is \cref{prop:cusp-forms-vanish-outside-relevant-locus}.
  To prove the necessity of condition (2),
  it suffices by \cref{eq:cusp-condition-summation-reordered},
  to find for each class $\cP$
  a parabolic line bundle $\pL = (\cL, I)$
  such that
  \begin{equation*}
    \Bunrel 1(\bF_q) \to \bZ, \qquad
    \pE \mapsto \# \homparinjsat((\cL, I), \pE)
  \end{equation*}
  is constant non-zero on $\cP$ and zero outside of $\cP$.
  This is in essence the same idea as in the proof
  of the vanishing outside the relevant locus
  (\cref{prop:cusp-forms-vanish-outside-relevant-locus}) but
  there we wanted the $\Ext^1( (\cM, D \setminus I), (\cL, I))$
  to be zero dimensional, whereas here we choose $(\cL, I)$
  such that the Ext-groups are one and two dimensional.

  The following choices of $\pL$ work.
  Let $x \in D$.
  For $\cP_x^{(1,0)}$,
  we can take $\pL = (\cO(1), \{x\})$.
  For $\cP_x^{(0,1)}$,
  we can take $\pL = (\cO, D \setminus \{x\})$.
  For $\cP_{\sigma}$,
  where $\sigma \colon \bP^1 \to \cohpar$ is the section of $\Supp \colon \cohpar \to \bP^1$
  such that for all $x \in D$,
  $\sigma(x) = k_x^{(1,0)}$,
  we can take
  $\pL = (\cO(1), \emptyset)$.
  This proves that condition (2.3) holds for every choice of $\sigma$
  (\Cref{rmk:cusp-conditions-vanish-on-every-p1-you-choose}).
\end{proof}

\begin{proof}[Proof of the sufficiency of the conditions in \cref{thm:cusp-condition-fq-pts-necessary-sufficient}]
  Using
  the fact that a parabolic morphism $(\cL, I) \to \pE$
  (with $\cL$ a line bundle)
  is the same as a morphism
  $\cL \to (T_I \pE)^0$,
  and using
  the inclusion-exclusion principle to remove maps that are not saturated in every parabolic degree,
  we can rewrite the cusp condition from \cref{eq:cusp-condition-summation-reordered}
  as
  \begin{equation*}
    \sum_{\pE \in \Bunrel 1(\bF_q)}
    \sum_{\emptyset \subseteq J \subseteq D}
    \frac{q - 1}{\#\Aut(\pE)} \# \homcohinjsat( \cL, (T_J T_I \pE)^0) \cdot f(\pE) = 0
  \end{equation*}
  where $\homcohinjsat$ denotes the set of injective saturated maps of coherent sheaves.

  It suffices to prove for all $J \subset D$, $I \subset D$ and $\cL$ a line bundle on $\bP^1$
  \begin{equation}
    \label{eq:sufficiency-theorem-cusp-condition-reduced-equation}
    \sum_{\pE \in \Bunrel 1(\bF_q)} \frac{q - 1}{\# \Aut(\pE)} \# \homcohinjsat(\cL, T_J T_I \cE)  \cdot f(\pE) = 0.
  \end{equation}

  Suppose that $\#J + \#I$ is even.
  Then for all $\pE \in \Bunrel 1(\bF_q)$,
  we have
  $T_J T_I \pE \in \Bunrel{d}$ for some odd $d \in \bZ$,
  so $(T_J T_I \pE)^0$ is isomorphic to $\cO(\floor{d/2}) \oplus \cO(\ceil{d/2})$.
  Therefore,
  there exists $c \in \bZ$ with
  \begin{equation*}
    c = \# \homcohinjsat(\cL, T_J T_I \cE) \qquad \text{for all $\pE \in \Bunrel 1(\bF_q)$}
  \end{equation*}
  and we can rewrite
  \eqref{eq:sufficiency-theorem-cusp-condition-reduced-equation}
  as
  \begin{equation*}
    c \sum_{\pE \in \Bunrel 1(\bF_q)} \frac{f(\pE)}{\# \Aut(\pE)(\bF_q)} = 0.
  \end{equation*}
  Our function $f$ does indeed satisfy this equation,
  because $\Bunrel 1(\bF_q)$ is a disjoint union of sets $\cP$ from
  (2.1), (2.2) and (2.3) in 
  \cref{thm:cusp-condition-fq-pts-necessary-sufficient}.

  Suppose that $d := \#J + \#I$ is odd.
  The reasoning is similar,
  but slightly more complicated,
  because there are two parabolic bundles with a different underlying vector bundle.
  However, these two bundles lie in a class $\cP_x^e$.
  More precisely,
  there exist
  $x \in D$, $e \in \{(1,0), (0,1)\}$
  and $\sigma \colon \bP^1 \to \cohpar$ a section of $\Supp \colon \cohpar \to \bP^1$
  that avoids $\alpha^{-1}(\cP_x^e) \subset \cohpar$,
  such that the following holds:
  letting
  \begin{equation*}
    \mathcal Q := \{\cP_\sigma\} \cup \{ \cP_x^\bullet\}_{x \in D} \setminus \{ \cP_x^e\}
  \end{equation*}
  we have a decomposition
  \begin{equation*}
    \Bunrel 1(\bF_q) = \cP_x^e \sqcup \bigsqcup_{\cP \in \mathcal Q} \cP
  \end{equation*}
  such that
  \begin{enumerate}
  \item for $\pE \in \cP_x^e$,
    the underlying bundle of $T_J T_I \pE$ is $\cO((3 - d)/2) \oplus \cO((-1 -d)/2)$; and
  \item
    for all $\pE \in \bigcup \mathcal Q$,
    the underlying bundle of $T_J T_I \pE$ is $\cO((1 - d)/2) \oplus \cO((1 - d)/2)$.
  \end{enumerate}
  Therefore,  there exist $c,d \in \bZ$ such that
  we can rewrite
  \eqref{eq:sufficiency-theorem-cusp-condition-reduced-equation}
  as
  \begin{equation*}
    c \sum_{\pE \in \cP_x^e} \frac{f(\pE)}{\# \Aut(\pE)(\bF_q)}
    + d \sum_{\pE \in \bigcup \mathcal Q} \frac{f(\pE)}{\# \Aut(\pE)(\bF_q)}
    = 0.
  \end{equation*}
  The function $f$ does indeed satisfy this equation by assumption.
\end{proof}

\section{Determining length 1 lower modifications}
\label{sec:calculation-of-all-length-1-lower-modifications}
To give a formula for the Hecke operators,
it is useful to first determine all the length 1 lower modifications of parabolic bundles in the relevant locus,
in terms of which the Hecke operators are defined.
It suffices to do this up to repeated application of the elementary Hecke operators $T_x \colon \Bunrel d \isom \Bunrel{d - 1}$ for $x \in D$.
We describe the parabolic bundles in $\Bunrel d$ using the chart
$T_\infty^{1 - d} \circ \alpha \colon \cohpar \isom \Bunrel d$
(\cref{defn:alpha}).

Every parabolic bundle in $\Bunrel d$
of the form $(T_\infty^{1 - d} \circ \alpha)(k_x^0)$ with $x \in D$
(see \cref{sec:moduli-spaces-of-parabolic-sheaves} for our notation of parabolic torsion sheaves)
can be obtained by applying elementary Hecke operators to $\pEtilde := (\cO(2), \emptyset) \oplus (\cO, D) \in \Bunrel 2$.
Similarly,
every parabolic bundle in $\Bunrel d$
of the form $(T_\infty^{1 - d} \circ \alpha)(k_x^{(1,0)})$
or $(T_\infty^{1 - d} \circ \alpha)(k_x^{(0,1)})$ with $x \in D$
can be obtained by applying elementary Hecke operators to $\pEhat \in \Bunrel 2$,
which was defined as the unique parabolic bundle in $\Bunrel 2$ with underlying vector bundle $\cO(2) \oplus \cO$
and automorphism group $\Gm$.
Hence,
we can describe the length one lower modifications of the parabolic bundles in
$\pi_d^{-1}(D)$ in terms of the length one lower modifications of $\pEtilde$ and $\pEhat$
and successive applications of the elementary Hecke operators.
We determine these modifications
\cref{sec:lower-modifications-of-petilde}
and
\cref{sec:lower-modifications-of-pehat}.
The remaining modifications
(of parabolic bundles in $\pi_d^{-1}(\bP^1 \setminus D)$)
are determined in \cref{sec:modifications-of-points-outside-of-d}.

\subsection{Length one lower modifications of $\pEtilde$}
\label{sec:lower-modifications-of-petilde}
We have already done most of this case in
\cref{prop:relevant-locus-fq-points-explicit-description}:
we determined that the generic modifications of $\pEtilde$
at different torsion sheaves $\pT \in \cohpar$
define different parabolic bundles in the relevant locus,
and we saw that moreover,
all parabolic bundles in the relevant locus are a generic modification of $\pEtilde$
at some $\pT \in \cohpar$.

It only remains to consider the non-generic modifications of $\pEtilde$.
These do not lie in the relevant locus:
it immediately follows from the definitions that
the flags in the closed orbits correspond to modifications with subbundle isomorphic to
either $(\cO(1), \emptyset) \oplus (\cO, D)$
or a parabolic bundle with underlying vector bundle $\cO(2) \oplus \cO(-1)$.

\subsection{Length one lower modifications of $\pEhat$}
\label{sec:lower-modifications-of-pehat}
\begin{thm}
  \label{thm:qualitative-description-of-modifications-hat-ce}
  ~\begin{enumerate}
  \item Let $x \in \bP^1 \setminus D$
    and consider the map
    \begin{equation*}
      \phi \colon
      \bP((\cO(2) \oplus \cO)|_x) \to \Bunpar^1,
      \qquad
      \ell \mapsto T_x^\ell \cEhat.
    \end{equation*}
    We decompose $\bP^1 = \bP((\cO(2) \oplus \cO)|_x)$ into the singleton
    $\{\cO(2)|_x\}$ and its complement $\bA^1 = \bP((\cO(2) \oplus \cO)|_x)$.
    \begin{enumerate}
    \item 
      The restriction of $\phi$ to $\bA^1$ factors through the relevant locus.
      The parabolic bundle $\phi(\cO(2)|_x)$ does not lie in the relevant locus.
    \item
      The map
      \begin{equation*}
        \pi_1 \circ \phi|_{\bA^1} \colon \bA^1 \to \bP^1 \setminus \{x\}
      \end{equation*}
      is an isomorphism.
    \item
      All parabolic vector bundles in the image of $\bA^1$ under $\phi$
      have automorphism group $\Gm$.
      % OLD explanation, using only words:
      % sends $\cO(2)|_x$ to $(\cO(2), \emptyset) \oplus (\cO(-1), D) \in \Bunpar^1 \setminus \Bunrel 1$.
      % The restriction to $\bA^1 \cong \bP((\cO(2) \oplus \cO)|_x) \setminus \{\cO(2)|_x\}$
      % factors through $\Bunrel 1$,
      % defining a degree 1 surjective map
      % \begin{equation*}
      %   \bA^1 \to \Bunrel 1 \setminus \pi_1^{-1}(\{x\}).
      % \end{equation*}
    \end{enumerate}
  \item
    Let $x \in D$.
    Recall that the modifications of $\pEhat = (\cO(2) \oplus \cO, (\ell_y)_{y \in D})$ at $x$ are classified by
    \begin{equation*}
        \bP^1(\cEhat^0|_x) \cup_{\ell_x} \bP^1((T_x^{\ell_x} \pEhat)^0|_x).
    \end{equation*}
    We can decompose this set according to the isomorphism class of
    the quotient of the corresponding modification:
    \begin{equation}
      \label{eq:thm-qualitative-description-of-modifications-hat-ce-decomposition}
      \begin{gathered}
        \bP^1(\cEhat^0|_x) \cup_{\ell_x} \bP^1((T_x^{\ell_x} \pEhat)^0|_x)\\
        =
        \underbrace{\bP^1(\cEhat^0|_x)\setminus\{\ell_x\}}_{k_x^{(0,1)}}
        \quad \sqcup \quad
        \underbrace{\{\ell_x\}}_{k_x^0}
        \quad \sqcup  \quad
        \underbrace{\bP^1((T_x^{\ell_x} \pEhat)^0|_x) \setminus \{\ell_x\}}_{k_x^{(1,0)}}.
      \end{gathered}
    \end{equation}
    % The modifications classified by $\bP^1(\cEhat^0|_x) \setminus \{\ell_x\}$
    % are the modifications with respect to $k_x^{(0,1)}$;
    % the modifications classified by $\bP^1((T_x \pEhat)^0|_x) \setminus \{\ell_x\}$
    % are the modifications with respect to $k_x^{(1,0)}$;
    % and the modification classified by $\ell_x$ is $T_x \cEhat$.
    Let
    \begin{equation*}
      \phi \colon \bP^1(\cEhat^0|_x) \cup_{\ell_x} \bP^1((T_x^{\ell_x} \pEhat)^0|_x) \to \Bunpar^1
    \end{equation*}
    denote the map that sends an element on the left to the corresponding modification of $\pEhat$.
    The following statements describe the restriction of $\phi$ to the disjoint sets in
    \eqref{eq:thm-qualitative-description-of-modifications-hat-ce-decomposition}.
    \begin{enumerate}
    \item
      There is a unique $\ell' \in \bP^1(\cEhat^0|_x) \setminus \{\ell_x\}$ such that
      the following restriction of $\phi$
      \begin{equation*}
        \bP^1(\cEhat^0|_x) \setminus \{\ell_x\} \xrightarrow{\phi} \Bunpar^1
      \end{equation*}
      is given by
      \begin{equation*}
        \ell \mapsto
        \begin{cases}
          (\cO(2), \emptyset) \oplus (\cO(-1), D) \not \in \Bunrel 1 & \text{if $\ell = \cO(2)|_x$} \\
          T_x \pEtilde & \text{if $\ell = \ell'$} \\
          T_x \pEhat & \text{otherwise}
        \end{cases}
      \end{equation*}
    \item
      The image of $\ell_x$ is $T_x \pEhat$.
    \item
      The restriction of $\phi$ to $\bP^1((T_x^{\ell_x} \cEhat)^0|_x) \setminus \{\ell_x\}$
      factors through
      \begin{equation*}
        \phi' \colon
        \bP^1((T_x^{\ell_x} \pEhat)^0|_x) \setminus \{\ell_x\}
        \to \Bunrel 1 \setminus \pi^{-1}(\{x\})
      \end{equation*}
      The map
      \begin{equation*}
        \pi_1 \circ \phi' \colon \bP^1((T_x^{\ell_x} \pEhat)^0|_x) \setminus \{\ell_x\}
        \to \bP^1 \setminus \{x\}
      \end{equation*}
      is an isomorphism.
      All parabolic bundles in the image of $\phi'$ have isomorphism group $\Gm$.
    \end{enumerate}
  \end{enumerate}
\end{thm}

\begin{notation}
  \label{not:subbundles-of-constant-sheaf}
  For the computational parts in the coming sections,
  we use the following notation.
  By $\cK$, \index{$K$@$\cK$}
  we denote the constant sheaf of rational functions functions on $\bP^1$.
  We fix a coordinate $X$ on $\bP^1$.
  Without mention to the contrary,
  we will consider the line bundle $\cO(n) = \cO(n [\infty])$ on $\bP^1$
  as the subsheaf of $\cK$
  of rational functions with a pole of order at most $n$ at $\infty$
  (or zeroes if $n$ is negative).
  Maps $\cO(n) \to \cO(m)$ with $m \geq n$
  are identified with elements of $\rH^0(\bP^1, \cO(m - n))$,
  i.e., degree $m - n$ polynomials in $X$,
  and similarly,
  maps between sub-vector bundles of $\cK^{\oplus 2}$
  can be denoted as matrices with coefficients in $\bF_q[X]$.

  Flags are denoted as follows.
  For $x \in \bP^1$, $\sigma, \tau \in \cK$ rational functions in $X$
  and $\cL, \cM \subset \cK$ line bundles,
  we denote by $(\sigma|_x : \tau|_x) \in \bP^1((\cL \oplus \cM)|_x)$
  the line generated by the pair of germs
  $(\sigma|_x, \tau|_x) \in \cL|_x \oplus \cM|_x$,
  provided that $\sigma$ and $\tau$ have appropriate poles and zeroes at $x$.
  As an example,
  for $r \in \bF_q$,
  $(r : 1) \in \bP^1((\cO(1) \oplus \cO)|_x)$ denotes a flag at $x$.
\end{notation}

\begin{proof}[Proof of \cref{thm:qualitative-description-of-modifications-hat-ce}]
  We take $\pEhat = (\cO(2) \oplus \cO, (\ell_p)_{p \in D})$
  where $\ell_p = \cO|_p$ for $p \in D \setminus \{0\}$
  and $\ell_0 = (1 : 1) \in \bP^1(\cO(2)|_0 \oplus \cO|_0)$.

  Let $x \in \bP^1(\bF_q) \setminus D$
  and let $\mu \in \bF_q \setminus \{1\}$.
  The underlying vector bundle of $\phi(\cO(2)|_x)$
  is $\cO(2) \oplus \cO(-1)$,
  so it does not lie in the relevant locus.
  The length 1 lower modification of $\pEhat$ at $x$
  with respect to the flag $(\mu x : 1) \in \bP^1((\cO(2) \oplus \cO)|_x)$
  is given by
  \begin{equation*}
    \begin{pmatrix}
      \lambda (X - x) & \mu X \\
      0 & 1
    \end{pmatrix} \colon
    (\cO(1) \oplus \cO, (\ell_p)_{p \in D \setminus \{t\}}, \ell'_t) \incl (\cO(2) \oplus \cO, (\ell_p)_{p \in D}) =: \pEhat
  \end{equation*}
  where
  $\lambda = \frac{1 - \mu}{1 - x}$
  and
  $\ell'_t = \left( - \frac{1 - x}{t - x} \frac{\mu t}{1 - \mu} : 1\right) \in \bP^1((\cO(1) \oplus \cO)|_t)$.
  Lastly,
  the modification with respect to $(x : 1) \in \bP^1((\cO(2) \oplus \cO)|_x)$
  is an inclusion of the bundle $(\cO(1) \oplus \cO, (\cO|_p)_{p \in D \setminus \{t\}}, (1 : 1) \in \bP^1((\cO(1) \oplus \cO)|_t)$.
  The subbundles of these inclusions all lie in the relevant locus and
  are pairwise distinct
  (i.e., different flags correspond to non-isomorphic sources),
  since the lines $\ell_t'$ are different
  and no non-trivial automorphism of $\cO(1) \oplus \cO$
  fixes $(\ell_p)_{p \in D \setminus \{t\}}$.
  We leave the rest of the proof of the first part,
  including the verification that $x$ is the point that does not lie in the image of $\pi_1 \circ \phi|_{\bA^1}$,
  to the reader.

  For the second part of the proposition,
  we only treat part of the case $x = t \in D$,
  the other cases being either very easy or similar.
  Here we see (using the same definitions of $(\ell_p)_{p \in D}$)
  that the length one lower modification given by the flag $\ell'_t := (v : 1) \in \bP^1((T_t \pEhat)^0|_t)$
  is the inclusion
  \begin{equation*}
    \begin{pmatrix}
      \frac{X - t}{1 - t} & 0 \\
      0 & 1
    \end{pmatrix}
    \colon
    (\cO(1) \oplus \cO, (\ell_p)_{p \in D \setminus \{t\}}, \ell'_t)
    \incl 
    (\cO(2) \oplus \cO, (\ell_p)_{p \in D}) =: \pEhat.
  \end{equation*}
\end{proof}

\subsection{Points outside of $D$}
\label{sec:modifications-of-points-outside-of-d}

The remaining case is the most difficult.

Let us first explain why we only 
need to consider modifications of parabolic bundles in $\pi_d^{-1}(\bP^1 \setminus D)$
with respect to torsion sheaves supported outside of $D$.
Suppose we wanted to know the length one lower modifications of $\pE \in \pi_d^{-1}(\bP^1 \setminus D)$
with respect to a torsion sheaf $\pT \in \cohpar$ supported at $D$.
Then by applying elementary Hecke operators to $\pE$,
we can assume $\pE \in \pi_1^{-1}(\bP^1 \setminus D)$.
That implies that $\pE$ is the generic length one lower modification of $\pEtilde$
at a point $y \in \bP^1 \setminus D$ (\cref{prop:relevant-locus-fq-points-explicit-description}).
Every length 1 lower modification of $\pE$ at $x \in D$
is therefore a length 1 lower modification of $\pEtilde$ at $y$
followed by a length 1 lower modification at $x$.
Changing the order of the modifications,
we get a length 1 lower modification of a length 1 lower modification of $\pEtilde$
at a point in $D$;
but this we have already determined in the previous sections.

The next theorem gives the remaining modifications.

\begin{thm}
  \label{thm:summary-statement-of-modifications-normal-points}
  % Let $(\cE, \cT) \in \Bunrelur 1 \times \barcohparopen$ be
  Let $(\pE, \pT) \in \Bunrel 1 \times \barcohpar$ be any point such that $\cT$ is supported outside of $D$
  and $\pE$ is a modification of $\pEtilde$ at a point outside of $D$.
  \begin{enumerate}
  \item 
    The map
    \begin{equation*}
      \begin{split}
        q|_{p^{-1}(\pE, \pT)} \colon p^{-1}(\pE, \pT) = \bP^1(\cE|_{\Supp \cT}) &\to \Bunpar^0,\\
      \ell &\mapsto T_{\Supp \cT}^\ell \pE
      \end{split}
    \end{equation*}
    factors through $\phi \colon \bP^1(\cE|_{\Supp \cT}) \to \Bunrel 0$.
  \item
    The composition of $\phi$ with the map $\pi_0 \colon \Bunrel 0 \to \bP^1$ to the coarse moduli space
    is a degree 2 map
    $\pi_0 \circ \phi \colon \bP(\cE|_{\Supp \cT}) \to \bP^1$.
  \item
    Let $\ell \in \bP^1(\cE|_{\Supp \cT})$.
    Then $\phi(\ell)$ has $\Gm \times \Gm$ as its automorphism group
    if and only if $(\pi_0 \circ \phi)(\ell)$ lies in $D$ and $\pi_0 \circ \phi$ is ramified at $\ell$.
  \end{enumerate}
\end{thm}

The following addendum uses the following convention.
Let $(\pE, \pT) \in \Bunrel 1 \times \barcohpar$ be as in the theorem,
i.e., $y := \pi_1(\pE)$ and $x := \Supp \cT$ lie in $\bP^1 \setminus D$.
Then up to isomorphism,
we can and do assume $\pE = T_y^{(1:1)} \pEtilde$,
so that we can consider $\pE$ as a subsheaf of $\cK^{\oplus 2}$
via the inclusions
$\pE \subset \pEtilde := (\cO(2), \emptyset) \oplus (\cO, D) \subset \cK^{\oplus 2}$.
Therefore, if $x \neq y$,
we can denote the flags in $\cE|_x$ by $(a:b)$ with $a,b \in \bF_q$ not both zero
(see \cref{not:subbundles-of-constant-sheaf}).
If $x = y$, then we use a uniformizer $\pi \in \cO|_y$
to write the flags as linear combinations of $(1:1)$ and $(1 + \pi : 1)$.
This is the notation that we use in the following addendum to the theorem.

\begin{addendum}
  \label{addendum:formula-to-theorem}
  Let $(\pE, \pT) \in \Bunrel 1 \times \barcohpar$
  and assume $y := \pi_1(\pE)$ and $x := \Supp \cT$ lie in $\bP^1 \setminus D$.
  If $x \neq y$,
  the degree 2 map in part (2) of 
  \cref{thm:summary-statement-of-modifications-normal-points}
  is given by the formula
  \begin{multline}
    \label{eq:formula-for-the-degree-2-map}
    (a : b)
    \mapsto \\
    \left(  \left(y a - x b \right)
      \cdot
      \left( (y - 1)(y - t) a - (x - 1)(x - t) b \right) : - (x - y)^2 ab \right),
  \end{multline}
  where in the target $\bP^1$, the point $(1:0)$ corresponds to $\infty \in D \subset \bP^1$.
  If $x = y$, 
  the map is given by
  \begin{equation}
    (s + r \pi : s) \mapsto
    ((ry - s)((y - 1)(y - t)r - (2y - (1 + t))s) : -s^2)
  \end{equation}
  where we use $\pi = (X - y) \in \cO|_y$ as a uniformizer.
  % NOTE/REFERENCE:
  % You get this by substituting
  % $a = s + r(y - x)$ and $b = s$
  % and then taking the limit for $y$ to $x$
\end{addendum}

We start the proof of
\cref{thm:summary-statement-of-modifications-normal-points}
with the following lemma.

\begin{lemma}
  \label{lem:modifications-of-normal-points-at-normal-points-always-in-relevant-locus}
  Let $x,y \in \bP^1 \setminus D$
  and let $\ell \subset (T_y^{(1:1)} \pEtilde)^0|_x$ be a flag.
  Then $T_x^\ell T_y^{(1:1)} \pEtilde$ lies in $\Bunrel 0$.
\end{lemma}
\begin{proof}
  It is equivalent to show that $T_\infty T_x^\ell T_y^{(1:1)} \pEtilde = T_x^\ell (T_\infty T_y^{(1:1)} \pEtilde)$
  lies in $\Bunrel{-1}$.
  The parabolic sheaf $T_\infty T_y^{(1:1)} \pEtilde$ is of the form $(\cO \oplus \cO, (\ell_p)_{p \in D})$.
  Because the underlying vector bundle is $\cO \oplus \cO$,
  we can naturally identify flags at different points.
  The modification $T_x^\ell T_\infty T_y^{(1:1)} \pEtilde$
  is of the form $(\cO \oplus \cO(-1), (\ell'_p)_{p \in D})$ for some flags $\ell'_p$
  and
  $\ell'_p = \cO|_p$ holds
  if and only if $\ell$ is equal to $\ell_p$
  (after identifying flags at different points).
  Moreover, the assumption $y \not \in D$ implies that the four flags $\ell_p$ are pairwise distinct.
  This proves there is at most one $p \in D$ with $\ell'_p = \cO|_p$.

  It remains to show that $T_\infty T_x^\ell T_y^{(1:1)} \pEtilde$ is not isomorphic to $(\cO, \emptyset) \oplus (\cO(-1), D)$.
  For this, we need that $y \not \in D$ implies that the cross-ratio
  $(\ell_\infty, \ell_0; \ell_1,  \ell_t)$ does not lie in $D$.
  It then follows that there is no inclusion 
  $(\cO, \emptyset) \oplus (\cO(-1), D) \incl T_\infty T_x^\ell T_y^{(1:1)} \pEtilde$.
  (More details are in \cite[Lemma 9.30]{mythesis}.)
\end{proof}

\begin{prop}
  \label{prop:special-preimages-of-degree-2-map}
  Let $x,y \in \bP^1 \setminus D$ with $x \neq y$.
  Write $\{0, 1, t\} = \{p_1, p_2, p_3\}$  and
  \begin{equation*}
    r = \frac{(x - p_1)(x - p_2)}{(y - p_1)(y - p_2)}
    \qquad \text{and} \qquad
    s = \frac{x - p_3}{y - p_3}.
  \end{equation*}
  Then the map
  \begin{equation}
    \phi \colon \bP((T_y^{(1:1)} \pEtilde)^0|_x) \to \Bunrel 1, \qquad \ell \mapsto T_x^\ell T_y^{(1:1)} \pEtilde
  \end{equation}
  sends the following flags to the following bundles:
  \begin{align*}
    && (1 : 0) &\mapsto \pEhat(-1) \\
    && (0 : 1) &\mapsto T_D \pEhat(1) \\
    \text{if $r = s$:}
    && (r : 1) &\mapsto T_{p_1} T_{p_2} \pEtilde (\cong T_{p_3} T_\infty \pEtilde) \\
    \text{if $r \neq s$:}
    && (r : 1) &\mapsto T_{p_1} T_{p_2} \pEhat  \\
    && (s : 1) &\mapsto T_{p_3} T_{\infty} \pEhat  
  \end{align*}
\end{prop} 
\begin{proof}
  Calculating the images of $(1:0)$ and $(0:1)$ is left to the reader.

  Let $x,y \in \bP^1 \setminus D$ with $x \neq y$
  and write $D = \{p_1, p_2, p_3, p_4\}$.
  Let
  \begin{equation*}
    \sigma \colon \cO \to \cO(2[\infty] - [p_1] - [p_2])
  \end{equation*}
  be the unique global section with $\sigma|_y = 1 \in \cO|_y$.
  Let $\cO^1, \cO^2$ be two copies of $\cO$.
  For $p = p_1, p_2$,
  we define the flag $\ell_p \subset (\cO^1 \oplus \cO^2)|_p$ by $\ell_p = \cO^2|_p$.
  Then there are flags $\ell_{p_3}, \ell_{p_4}$ at $p_3$ and $p_4$
  that are different from $\cO^2|_{p_3}$ and $\cO^2|_{p_4}$,
  such that the inclusion
  \begin{equation*}
    \begin{pmatrix}
      (X - x)(X - y) & \sigma \\ 0 & 1
    \end{pmatrix}
    \colon
    (\cO^1 \oplus \cO^2, (\ell_p)_{p \in D}) 
    \incl \pEtilde
  \end{equation*}
  has image $T_x^{(\sigma|_x : 1)} T_y^{(1 : 1)} \pEtilde \subset \pEtilde$.
  Note that $\sigma|_x = r$ if $\infty \not \in \{p_1, p_2\}$
  and $\sigma|_x = s$ if $\infty \in \{p_1, p_2\}$,
  so this is indeed the modification we want to determine.
  If $\ell_{p_3} = \ell_{p_4}$,
  then $(\cO^1 \oplus \cO^2, (\ell_p)_{p \in D}) $
  is isomorphic to $T_{p_1} T_{p_2} \pEtilde$;
  otherwise,
  it is isomorphic to $T_{p_1} T_{p_2} \pEhat$.
  A calculation shows $\ell_{p_3} = \ell_{p_4}$ is equivalent to $r = s$.
  \end{proof}

Let $x, y \in \bP^1 \setminus D$.
The isomorphism
\begin{equation*}
  \begin{pmatrix}
    0 & 1 \\
    \frac{y(y - 1)(y - t)}{X(X - 1)(X - t)} & 0
  \end{pmatrix}
  \colon
  T_D \pEtilde(2) \isom \pEtilde
\end{equation*}
induces isomorphisms
\begin{equation}
  T_D T_y^{(1:1)} \pEtilde(2) \isom T_y^{(1:1)} \pEtilde,
\end{equation}
\begin{equation}
  \label{eq:td-symmetry-fixes-almost-everything}
  T_D T_\infty^d T_y^{(1:1)} \pEtilde(2) \isom T_\infty^d T_y^{(1:1)} \pEtilde \qquad (d \in \bZ)
\end{equation}
and
\begin{equation}
  \label{eq:td-symmetry-swaps-flags}
  T_D T_x^{(r:1)} T_y^{(1:1)} \pEtilde(2)
  \isom T_x^{(x(x-1)(x-t) : y (y - 1) (y - t) \cdot r)} T_y^{(1 : 1)} \pEtilde.
\end{equation}
\Cref{eq:td-symmetry-fixes-almost-everything}
shows that the operation $T_D(2) \colon \pF \mapsto T_D \pF(2)$
fixes the isomorphism classes in $\pi_d^{-1}(\bP^1 \setminus D)$
for all $d \in \bZ$.
Every $\pE \in \Bunrel d$ with $\Aut \pE \cong \Gm \times \Gm$
can be obtained from $\pEtilde$
by applying elementary Hecke operators,
and because the elementary Hecke operators commute with $T_D(2)$,
we find that $T_D\pE(2)$ is also isomorphic to $\pE$.
In other words,
all bundles in $\Bunrel d$ with the exception of those in $\pi_d^{-1}(D)$
with automorphism group $\Gm$,
are fixed by $T_D(2)$.
\Cref{eq:td-symmetry-swaps-flags}
shows that if $T_x^{(r:1)} T_y^{(1:1)} \pEtilde$
is fixed by $T_D(2)$,
then $\phi(r:1) = \phi(x(x-1)(x-t) : y (y - 1) (y - t) \cdot r)$.
The following proposition shows that 
these are in fact the only flags that $\phi$ maps to the same parabolic bundles.
\begin{prop}
  \label{prop:degree-2-map-images-are-different-iff}
  Let $\phi, x, y$ as in 
  \cref{prop:special-preimages-of-degree-2-map}.
  Let $r, s \in \bF_q$ with $r \neq s$.
  Then there is an isomorphism
  \begin{equation}
    \label{eq:degree-2-map-images-are-different-the-isom}
    \phi(r : 1)
    = T_x^{(r:1)} T_y^{(1:1)} \pEtilde
    \isom
    T_x^{(s:1)} T_y^{(1:1)} \pEtilde = \phi(s : 1)
  \end{equation}
  if and only if
  \begin{align}
    \label{eq:degree-2-map-images-are-different-cond-1}
    rs &= \frac{x(x-1)(x-t)}{y(y-1)(y-t)}
    && \\
    \intertext{and}
    \label{eq:degree-2-map-images-are-different-cond-2}
    r,s &\neq \frac{x - p}{y - p}
    && \text{for all $p \in \{0,1,t\}$}.
  \end{align}
  Every such isomorphism is given on the underlying vector bundle
  by a non-zero multiple of the map
  \begin{equation*}
    \label{eq:degree-2-map-images-are-different-the-matrix-for-the-isom}
    \frac{1}{(X - x)(X - y)}
    \begin{pmatrix}
      a & - \frac{X(X - 1)(X - t)}{y(y-1)(y-t)}  \\
      1 & d
    \end{pmatrix}
  \end{equation*}
  where $a$ and $d$ are the unique degree 2 polynomials in $X$ satisfying
  the following equations
  \begin{align*}
    a(x) &= s &
                d(x) &= -r \\
    a(y) &= 1 &
                d(y) &= -1 \\
    a'(y) - d'(y) &= -b'(y) &
                              ra'(x) - sd'(x) &= - b'(x)
  \end{align*}
\end{prop}
\begin{proof}
  Both parabolic sheaves
  $T_x^{(r:1)} T_y^{(1:1)} \pEtilde$
  and
  $T_x^{(s:1)} T_y^{(1:1)} \pEtilde$
  are by definition subsheaves of $\pEtilde = (\cO(2), \emptyset) \oplus (\cO, D)$.
  We consider the underlying vector bundles as subbundles of $\cK^{\oplus 2}$,
  where $\cK$ is the constant sheaf of rational functions on $\bP^1$.
  Any map
  from
  $T_x^{(r:1)} T_y^{(1:1)} \pEtilde$
  to
  $T_x^{(s:1)} T_y^{(1:1)} \pEtilde$
  is induced by a map
  \begin{equation}
    \label{eq:pf-of-degree-2-map-on-sheaf-of-rational-functions}
    \frac 1 {(X - x)(X - y)}
    \begin{pmatrix}
      a & b \\ c & d
    \end{pmatrix} \colon \cK^{\oplus 2} \to \cK^{\oplus 2}
  \end{equation}
  with
  $a \in \rH^0(\bP^1, \cO(2))$,
  $b \in \rH^0(\bP^1, \cO(4))$,
  $c \in \rH^0(\bP^1, \cO)$
  and $d \in \rH^0(\bP^1, \cO(2))$.

  The idea of the proof is to
  determine when a map as in equation
  \eqref{eq:pf-of-degree-2-map-on-sheaf-of-rational-functions}
  defines an isomorphism
  $T_x^{(r:1)} T_y^{(1:1)} \pEtilde \isom T_x^{(s:1)} T_y^{(1:1)} \pEtilde$.
  The following three conditions are necessary and sufficient:
  (1) the map \eqref{eq:pf-of-degree-2-map-on-sheaf-of-rational-functions}
  defines a map on the underlying vector bundles
  $(T_x^{(r:1)} T_y^{(1:1)} \pEtilde)^0 \isom (T_x^{(s:1)} T_y^{(1:1)} \pEtilde)$;
  (2) the map on the underlying vector bundles respects the parabolic structure;
  and 
  (3) the map \eqref{eq:pf-of-degree-2-map-on-sheaf-of-rational-functions}
  is injective.

  The last condition is sufficient to ensure that the map is an isomorphism.
  The second condition is equivalent to the condition that $b$ has zeroes at $D$,
  while $d$ does not.
  This determines $b$ up to a scalar multiple.

  To continue, we can first show that $c \neq 0$ is necessary,
  so that after scaling, we can assume $c = 1$.
  Then the first condition leads to the conditions on $a$ and $d$ stated in the theorem.
  Some calculations show that these have solutions precisely in the cases stated in the theorem.
 \end{proof}

 \begin{proof}[Proof of \cref{thm:summary-statement-of-modifications-normal-points}
   and \cref{addendum:formula-to-theorem}]
   \label{proof-of-thm:summary-statement-of-modifications-normal-points}

   The first part of the theorem
   (the modifications lie in the relevant locus)
   is exactly the statement of \cref{lem:modifications-of-normal-points-at-normal-points-always-in-relevant-locus}.
   To prove the second part,
   note that $\pi_0 \circ \phi \colon \bP^1 \to \bP^1$ has degree 2 by \cref{prop:degree-2-map-images-are-different-iff}
   (the proof for $x = y$ is similar).
   The formulas for the map $\pi_0 \circ \phi$ in the addendum
   can be checked by using that it is a degree 2 map and checking the flags that map to $D$
   (which we calculate in \cref{prop:special-preimages-of-degree-2-map} for $x \neq y$;
   the formula for $x = y$ can be obtained as a limit).
   Lastly,
   the ramification behaviour stated in the third part of the theorem
   follows from \cref{prop:special-preimages-of-degree-2-map},
   where we found some flags that map to $\pi_0^{-1}(D)$;
   we now know that these are all preimages of the points in $\pi_0^{-1}(D)$.
   %
   % EARLIER ATTEMPTS before I realized how easy it was
   % if ramified over $x \in D$,
   % then the image of $\phi$ is a point with $\Gm \times \Gm$-automorphisms
   % (because $T_D(2)$ does not fix)
   % then the equations show that 
   %
   % \innote{Rewrite this paragraph}
   % if ramified,
   % then $r = s$,
   % and since $T_D of r$ = $s$,
   % it is fixed by $T_D(2)$, so it is the point with automorphism group $\Gm \times \Gm$.
   % 
   % The converse
   % \innote{can be shown from the structure of the moduli stack? Yes, I think so again,
   % again by looking at the tangent vectors.}
   %
   % This tangent space thing does indeed work:
   % use that the tangent vectors at (0,0) have to be Gm-invariant too,
   % so they should lie along one of the axes
\end{proof}

\section{Determining the Hecke operators}
\label{sec:determining-hecke-operators}
\subsection{Basis of the cusp forms}
\label{sec:basis-of-cusp-forms-on-bunrel}
We first define a basis $\{F_z^d\}_{z \in \bF_q}$ of cusp forms on $\Bunpar^d$.
It is immediate from our characterization of the cusp forms on $\Bunpar^d(\bF_q)$
(\cref{thm:cusp-condition-fq-pts-necessary-sufficient})
that this is indeed a basis.

\begin{defin}
  \label{defn:cusp-forms-fz}
  Let $z \in \bF_q$.
  We denote by \index{$F_z$}
  \begin{equation*}
    F_z^1 = F_z \colon \Bunpar^1(\bF_q) \to \bQ_\ell
  \end{equation*}
  the unique cusp form satisfying
  for all $y \in \bF_q$
  \begin{equation*}
    F_z(T_y^{(1:1)} \pEtilde)
    =
    \begin{cases}
      1 & \text{if $y = z$} \\
      0 & \text{otherwise}
    \end{cases}.
  \end{equation*}
  For $d \in \bZ$,
  we denote by
  \index{$F_z^d$}
  $F_z^d$ the composition
  \begin{equation*}
    F_z^d := F_z \circ T_\infty^{d - 1} \colon \Bunpar^d(\bF_q) \to \bQ_\ell.
  \end{equation*}
\end{defin}

\begin{remark}
  \label{rmk:properties-of-basis-cusp-form}
  Let $z \in \bF_q$.
  The cusp form $F_z$ is supported on $\{\pi_1^{-1}(z), \pi_1^{-1}(\infty)\}$.
  It satisfies $F_z(T_\infty^{(1:1)} \pEtilde) = -1$
  (\cref{thm:cusp-condition-fq-pts-necessary-sufficient}, condition (2.3))
  % \cref{item:characterisation-cusp-forms-vanish-on-p1s}
  and for all $x \in D$,
  $F_z(T_x^{(1:1)} \pEtilde) = F_z(\leftmodif{(1:1)}{x} \pEtilde) = - \frac{1}{q - 1}F_z(T_x \pEtilde)$
  (\cref{thm:cusp-condition-fq-pts-necessary-sufficient}, conditions (2.1) and (2.2)).
  % \cref{item:characterisation-cusp-forms-vanish-on-a1s-second}
  % and
  % \cref{item:characterisation-cusp-forms-vanish-on-a1s-first}
\end{remark}

% \begin{remark}
%   For every $d \in \bZ$,
%   the $F_z^d$ with $z \in \bF_q$
%   form a basis of the cusp forms on $\Bunpar^d(\bF_q)$.
% \end{remark}

\subsection{Formula for the Hecke operators}
\label{sec:the-formula-for-hecke-operators}

Let $x \in D$.
We denote by $M_x \colon \bP^1 \isom \bP^1$ the unique M\"obius transformation
that preserves $D$ and sends $\infty$ to $x$.
We note that for $x, y \in D$,
we have $M_x(y) = M_y(x)$.

\begin{thm}
  \label{thm:formula-hecke-operators}
  Let $z \in \bF_q$ and let $x \in \bP^1$.
  Let $\pT \in \cohpar$ be a parabolic torsion sheaf supported at $x$
  with automorphism group $\Gm$ 
  (automatic if $x \not \in D$)
  and let $\heckelocal x$ be the Hecke operator with respect to $\pT$.
  First suppose $x \neq \infty$.
  Then
  \begin{equation*}
    \heckelocal x F_z^0 = \sum_{y \in \bF_q} \alpha_{z,y}^x F_y
  \end{equation*}
  where
  for all $y \in \bF_q \setminus \{x\}$,
  \begin{equation*}
    \begin{split}
      \alpha_{z,y}^x &=
      \# \left\{
        r \in \bF_q^* :
        z = \frac{(yr - x)((y - 1)(y - t)r - (x - 1)(x - t))}{- (x - y)^2 r}
      \right\} 
      \\
      &\phantom{=}-
      \begin{cases}
        0 & \text{if $x \in D$ and $y \in D$} \\
        1 & \text{if $x \in D$ or $y \in D$, but not both} \\
        2 & \text{otherwise}
      \end{cases}
      \\
      &\phantom{=}-
      \begin{cases}
        q & \text{if $z \in D$ and $y = M_z(x)$} \\
        0 & \text{otherwise}
      \end{cases}
    \end{split}
  \end{equation*}
  and
  \begin{equation*}
    \alpha_{z,x}^x
    = \# \left\{
      r \in \bF_q :
      z = -(yr - 1)((y - 1)(y - t)r - (2y - (1+ t)))
    \right\}
    - q + 1.
  \end{equation*}
  If $x = \infty$,
  then the same holds with
  \begin{equation*}
    \alpha_{z,y}^\infty
    =
    \begin{cases}
      -1 & \text{if $z = y$} \\
      0 & \text{otherwise}
    \end{cases}.
  \end{equation*}
\end{thm}
\begin{proof}
  By definition of the basis $\{F_{z'}\}_{z' \in \bF_q}$,
  we have
  \begin{equation*}
    \alpha_{z,y}^x = (\heckelocal x F_z^0)(T_y^{(1:1)} \pEtilde).
  \end{equation*}
  This in turn is by definition of $\heckelocal x$ and of $F_z^0$
  \begin{equation*}
    \alpha_{z,y}^x
    =
    \sum_{\pF \subset_{\pT} T_y^{(1:1)} \pEtilde} F_z(T_\infty^{-1} \pF).
  \end{equation*}
  where the summation is
  over length 1 lower modifications $\pF \subset T_y^{(1:1)} \pEtilde$
  with cokernel $\pT$.
  In \cref{sec:calculation-of-all-length-1-lower-modifications},
  we have calculated all these modifications,
  so that the proof of this theorem is reduced to careful bookkeeping with the results from that section.
  
  The term
  $\# \left\{r \in \bF_q^* : z = \frac{(yr - x)((y - 1)(y - t)r - (x - 1)(x - t))}{- (x - y)^2 r} \right\}$ 
  that appears in the formula for $\alpha_{z,y}^x$,
  counts
  in the generic case
  ($x,y \not \in D$ and $x \neq y$; \cref{thm:summary-statement-of-modifications-normal-points})
  the lines $(r:1) \in \bP^1((T_y^{(1:1)} \pEtilde)^0|_x)$
  whose corresponding modifications have subbundle $T_\infty T_z^{(1:1)} \pEtilde$.
  (Note that the lines $(1:0) = \cO(1)|_x$ and $(0:1) = \cO|_x$ in $\bP^1((T_y^{(1:1)} \pEtilde)^0|_x)$
  are mapped to $\pEhat(-1)$ or $\pEhat(-1)$.)

  The other terms provide correction terms for the non-generic cases
  and also take into account the contributions of the points $\pF \in \pi_0^{-1}(\infty)$.
\end{proof}

\section{$E$ is its own eigenfunction; argument on $\bF_q$-points}
\label{sec:argument-on-function-level}

Here we prove on the level of $\bF_q$-points that
for any Hecke eigenfunction $f \colon \Bunpar(\bF_q) \to \bQ_\ell$,
after scaling $f$ so that $f(\pEtilde) = q - 1$,
the eigenvalue for the Hecke operator corresponding to $\pT \in \cohpar(\bF_q)$
is equal to $f(\alpha(\pT))$,
where $\alpha \colon \cohpar \isom \Bunrel 1$ is the isomorphism given in \cref{defn:alpha}.
In the following sections, we prove the analogous statement on the geometric level (\cref{thm:intro-eigensheaf}),
which implies the statement given in this section.
We only include this section,
because the proof is short and enlightening
and showcases some of the ideas that will play a role in the next sections.

\begin{prop}
  Let $f \colon \Bunpar(\bF_q) \to \bQ_\ell$
  be a cusp form and a Hecke eigenfunction with eigenvalues
  $(\lambda_y)_{y \in \bP^1 \setminus D}$ for the operators $\heckelocal y$ with $y \in \bP^1 \setminus D$,
  eigenvalues $(\lambda_y)_{y \in D}$ for the operators $\heckelocall y$ with $y \in D$
  and
  eigenvalues $(\lambda_y')_{y \in D}$ for the operators $\heckelocalr y$ with $y \in D$.
  Then after scaling $f$, we have for all $y \in \bP^1 \setminus D$
  \begin{equation*}
    f(\alpha(k_y)) = \lambda_y
  \end{equation*}
  and for all $y \in D$
  \begin{equation*}
    f(\alpha(k_y^{(0,1)})) = f(\alpha(k_y^{(1,0)})) = \lambda_y = \lambda_y'.
  \end{equation*}
\end{prop}
\begin{proof}
  Let $y \in \bP^1 \setminus D$.
  By the definition of $\lambda_y$ and of $\heckelocal y$,
  we have
  \begin{equation}
    \label{eq:prop-argument-on-fq-pts-proof-definitions}
    \lambda_y f(\pEtilde)
    = (\heckelocal{y} f)(\pEtilde)
    = \sum_{\mathclap{\pF \incl \pEtilde \surj k_y}} f(\pF)
  \end{equation}
  where the sum is over all length one lower modifications $\pF$ of $\pEtilde$
  with respect to $k_y$.
  We recall that these modifications are classified by flags of $\cEtilde = \cO(2) \oplus \cO$
  at $y$ (\cref{sec:relevant-locus})
  and that the automorphisms of $\pEtilde$ act on these flags.
  This action has three orbits: the singleton orbits $\{\cO|_y\}$ and $\{\cO(2)|_y\}$,
  and the generic orbit.
  The first two orbits correspond to modifications $\pF$ that do not lie in the relevant locus
  $\Bunrel 1$,
  and hence the cusp form $f$ vanishes at those modifications.
  The remaining, generic orbit corresponds to modifications of the form $\alpha(k_y) \subset \pEtilde$.
  Hence, the sum
  in \cref{eq:prop-argument-on-fq-pts-proof-definitions}
  is equal to $(q - 1) f(\alpha(k_y))$,
 which proves
 \begin{equation*}
   f(\alpha(k_y)) = \frac{f(\pEtilde)}{q - 1} \cdot \lambda_y.
 \end{equation*}
 For $y \in D$, we prove in a similar way that $f(\alpha(k_y^{(1,0)})) = \lambda_y$
 and $f(\alpha(k_y^{(0,1)})) = \lambda_y'$;
the cusp conditions (\cref{thm:cusp-condition-fq-pts-necessary-sufficient}; more specifically, the combination of (2.1) and (2.2))
then imply that these are equal.
\end{proof}

\section{Cohomological properties of $E$}
\label{sec:cohomological-properties-of-e}

In this section,
we deduce some cohomological properties of the intermediate extension of the irreducible rank 2 pure local system $E$
with unipotent monodromy
along the inclusions $\bar j \colon \bP^1 \setminus D \incl \bP^1$
and $j \colon \bP^1 \setminus D \incl \bP^1$.
Recall that there are $2^4$ embeddings of $\bP^1$ into $\barcohpar$
that are compatible with the inclusions.
Note also that since $\bar j$ is an open embedding into a curve,
we have $\bar j_{!*} = \bar j_*$.

\begin{lemma}
  \label{lem:euler-characteristic-of-e-is-zero}
  The Euler-characteristic of $\bar j_{!*} E$ is zero.
\end{lemma}
\begin{proof}
  The Grothendieck-Ogg-Shafarevich formula
  (\cite[formula 7.2]{sga5}, or \cite[theorem  9.1]{kindlerruelling2014})
  % References:
  % [SGA5, Exp. X], [Kat88, Ch. 2], see also [Ray95].
  % Exp. X formula 7.2
  says
  \begin{equation*}
    \chi_\mathrm c(\bP^1_{\bar k} \setminus D, E)
    = \rank(E) \cdot \chi_\mathrm c(\bP^1_{\bar k} \setminus D, \bQ_\ell)
    - \sum_{x \in D} \mathrm{Sw}_x(E),
  \end{equation*}
  where $\chi_\mathrm c$ is the alternating sum of the dimension of the cohomology groups with compact support
  and $\mathrm{Sw}_x(E)$ is the Swan conductor,
  which is zero because $E$ is tamely ramified.
  Note that
  \begin{equation*}
    \chi_\mathrm c(\bP^1_{\bar k} \setminus D, E)
    = \chi(\bR \bar j_! E).
  \end{equation*}
  Applying the additivity of the Euler characteristic to the distinguished triangle
  \begin{equation*}
    \bR \bar j_! E \to \bar j_{!*} E \to (\bar j_{!*} E)|_D \xrightarrow{+1},
  \end{equation*}
  we find
  \begin{equation*}
    \begin{split}
      \chi(\bP^1, \bar j_{!*} E)
      &= \chi(\bR \bar j_! E)  + \chi( (\bar j_{!*} E)|_D) \\
      &= \rank(E) \cdot \chi_\mathrm c(\bP^1_{\bar k} \setminus D, \bQ_\ell)
      + \chi( (\bar j_{!*} E)|_D) \\
      &= 2 \cdot (-2) + 4 = 0.
    \end{split}
  \end{equation*}
\end{proof}

\begin{prop}
  \label{prop:cohomology-groups-of-intermediate-extension-are-zero}
  All cohomology groups $\rH^i(\bP^1, \bar j_{!*} E[1])$ are zero.
\end{prop}
\begin{proof}
  Because $E$ and its dual are irreducible,
  $\rH^{-1}(\bP^1, \bar j_{!*} E)$ and $\rH^1(\bP^1, \bar j_{!*} E)$ vanish.
  Since the Euler characteristic of $\bar j_{!*} E$ is zero
  (\cref{lem:euler-characteristic-of-e-is-zero}),
  $\rH^0(\bP^1, \bar j_{!*} E)$ also vanishes.
\end{proof}

Let $x \in D$.
Recall from
\cref{sec:moduli-spaces-of-parabolic-sheaves}
our notation
$k_x^{(1,0)}$,
$k_x^{(0,1)}$,
and
$k_x^{0}$
for three parabolic torsion sheaves in $\cohpar$
representing the three distinct isomorphism classes of sheaves in $\cohpar$ supported on $x$.
\begin{lemma}
  \label{lem:stalks-of-intermediate-extension-of-e}
  Let $x \in D$.
  \begin{enumerate}
  \item 
    The stalks of $j_{!*}E[1]$ at the torsion sheaves supported at $x$ are given by
    \begin{align*}
      (j_{!*}E[1])|_{k_x^{(1,0)}}
      &= (j_{!*}E[1])|_{k_x^{(0,1)}}
        = (\bar j_* E[1])|_x \\
      \intertext{and}
      (j_{!*} E[1])|_{k_x^0}
      &= (\bar j_* E[1])|_x \otimes \rH^*(\Gm, \bQ_\ell).
    \end{align*}
  \item
    Let $\pT \in \cohpar$ be one of the two parabolic length 1 torsion sheaves with $\Gm$-automorphisms
    that are supported on $x \in D$.
    Let
    \begin{equation*}
      \cP_x := \{k_x^0, \pT\} \subset \barcohpar
    \end{equation*}
    the substack that contains $k_x^0$ and $\pT$.
    Then
    \begin{equation}
      \rH_{\mathrm c}^*(\cP_x, (j_{!*} E)|_{\cP_x}) = 0.
    \end{equation}
  \end{enumerate}

\end{lemma}
\begin{proof}
  This follows from \cite[corollary 4.5]{heinloth2004}
  (in particular the first few sentences of the proof).
  % For notes on the proof,
  % see \verb|20180913_173248-spectral-sequence.jpg|.
\end{proof}

\begin{remark}
  \label{rmk:cohomological-property-of-e-via-maps}
  Let $x \in D$.
  Consider the map
  \begin{equation*}
    k_x^{(\dash, 0)} \colon \bA^1 \to \barcohpar,
    \qquad \lambda \mapsto (\mathellipsis \to k_x \xrightarrow{\lambda} k_x \xrightarrow{0} k_x \to \mathellipsis)
  \end{equation*}
  whose image in $\barcohpar$ contains exactly the points $k_x^{(1,0)}$ and $k_x^0$,
  and also consider the analogous map
  \begin{equation*}
    k_x^{(0, \dash)} \colon \bA^1 \to \barcohpar,
    \qquad \lambda \mapsto (\mathellipsis \to k_x \xrightarrow{0} k_x \xrightarrow{\lambda} k_x \to \mathellipsis).
  \end{equation*}
  The statements in
  \cref{lem:stalks-of-intermediate-extension-of-e}
  follow from the more explicit formula
  \begin{equation*}
    (k_x^{(\dash, 0)})^* (j_{!*} E)
    = (k_x^{(0, \dash)})^* (j_{!*} E)
    = (\bar j_{!*} E)|_x \otimes \bR (\Gm \incl \bA^1)_* \bQ_\ell.
  \end{equation*}
\end{remark}

\section{Definition and perversity of the Hecke eigensheaf}
\label{sec:definition-and-perversity-of-hecke-eigensheaf}

Here we define the proposed Hecke eigensheaf $\Aut_E$ associated
to the irreducible pure rank 2 local system $E$
and prove that it is irreducible and perverse.
The proof that it is in fact the Hecke eigensheaf associated to $E$
is \cref{thm:aute-is-indeed-hecke-eigensheaf}.

Recall that we have a natural inclusion $j \colon \bP^1 \setminus D \incl \barcohpar$
and a canonical isomorphism $\alpha \colon \cohpar \isom \Bunrel 1$.
Since $\barcohpar$ is the $\Gm$-rigidification of $\cohpar$
and $\Gm$-actions on $\ell$-adic sheaves are necessarily trivial,
the category of $\ell$-adic sheaves on $\barcohpar$ is equivalent to the category of $\ell$-adic sheaves
on $\cohpar$.
We write
\begin{equation}
  \cL_E := j_{!*} E
\end{equation}
for the complex of sheaves on $\barcohpar$
and use the same notation for its pullback to $\cohpar$.

We define $\Aut_E \in \boundedderivcat(\Bunpar, \bQ_\ell)$ by first defining its restriction
$\Aut_E^1 := \Aut_E|_{\Bunpar^1}$
as the extension by zero of $\alpha_* \cL_E$ along $\Bunrel 1 \incl \Bunpar^1$.
If $\Aut_E$ is to be a Hecke eigensheaf,
it is in particular an eigensheaf for the elementary Hecke operator $T_\infty$;
we use this to define the restriction $\Aut_E^d := \Aut_E|_{\Bunpar^d}$ for $d \in \bZ$.
\index{$AutEd$@$\Aut_E^d$}
This leads to the following definition.
We denote by $\Ecoeff$ the constant local system whose fiber is the fiber of $j_{!*} E$ at $k_{\infty}^{(1,0)}$.

\begin{defin}
  \label{defn:aut-e}
  For $E$ a pure irreducible rank 2 local system on $\bP^1 \setminus D$ with unipotent monodromy,
  we define $\Aut_E \in \boundedderivcat(\Bunpar, \bQ_\ell)$ as the complex that is zero outside of
  \index{$AutE$@$\Aut_E$}
  $\Bunrel \bullet \subset \Bunpar$
  and that satisfies
  \begin{equation*}
    \Aut_E|_{\Bunrel d} = (T_\infty^{d - 1})^* (\alpha_* \cL_E[1]) \otimes \Ecoeff^{\otimes -d + 1}.
  \end{equation*}
\end{defin}

The main result of this section is the following theorem.
\begin{thm}
  \label{thm:irreducibility-perversity-of-aute}
  For every $d \in \bZ$,
  the complex $\Aut_E|_{\Bunpar^d}$ is an irreducible perverse sheaf.
\end{thm}
We conclude the proof of this theorem on page
\pageref{proof-of-thm:irreducibility-perversity-of-aute}.

We denote by $\ave$ the functor
\begin{equation*}
  \ave := \bR q_! p^* \colon \boundedderivcat(\Bunpar^2 \times \barcohpar, \bQ_\ell) \to \boundedderivcat(\Bunpar^1, \bQ_\ell)
\end{equation*}
where $p,q$ are the maps in the Hecke correspondence
\begin{equation*}
  \Bunpar^2 \times \barcohpar \xleftarrow{p} \cH \xrightarrow{q} \Bunpar^1.
\end{equation*}

% \innote{Hm, I could rewrite this as Fourier transforms... Maybe that would be better?}

We denote by $\BAut(\pEtilde) \subset \Bunpar^2$ the residual gerbe associated to $\pEtilde \in \Bunpar^2(\bF_q)$
and denote by
\begin{equation*}
  i \colon \BAut(\pEtilde) \times \barcohpar \incl \Bunpar^2 \times \barcohpar 
\end{equation*}
the inclusion.
By $\tildeLaumon$ we denote the pullback of $\cL_E$ along
the projection $\BAut(\pEtilde) \times \barcohpar \to \barcohpar$.

\begin{lemma}
  \label{lem:aute1-via-fourier-transform}
  We have
  $\Aut_E^1 = \ave (\bR i_! \tildeLaumon[1])$.
\end{lemma}
\begin{proof}
  Let $\cHcE \subset \cH$ be 
  the substack defined as the pullback of $p \colon \cH \to \Bunpar^2 \times \barcohpar$
  along the inclusion $i$
  and let $p_{\pEtilde}$ and $q_{\pEtilde}$ denote the maps following maps in the commutative diagram
  \begin{equation*}
    \begin{tikzcd}
      \BAut(\pEtilde) \times \barcohpar \ar[d, incl, "i"]
      & \cHcE \ar[l, "p_{\pEtilde}"] \ar[rd, "q_{\pEtilde}"] \ar[d, incl] \ar[ld, phantom, "\square"]
      & \\
      \Bunpar^2 \times \barcohpar
      & \cH \ar[l, "p"] \ar[r, swap, "q"]
      & \Bunpar^1
    \end{tikzcd}
  \end{equation*}
  where the left square is Cartesian.
  By applying the proper base change theorem to the left square,
  we find
  \begin{equation}
    \label{eq:geometric-chapter-aut-e-1-via-proper-base-change}
     \ave \bR i_! \tildeLaumon = \bR q_{\pEtilde, !} p_{\pEtilde}^* \tildeLaumon.
  \end{equation}

  In \cref{sec:identifying-cohpar-with-bunrel},
  we defined a substack $\cHcErel \subset \cH$;
  we now note that in fact $\cHcErel = q_{\pEtilde}^{-1}(\Bunrel 1) \subset \cHcE$.
  We showed that the restrictions of $q_{\pEtilde}$ to $\cHcErel$ is an isomorphism
  and that the restriction of $p_{\pEtilde}$ to $\cHcErel$
  composed with
  $\BAut(\pEtilde) \times \barcohpar \to \BAut(\pEtilde)/\Gm \times \barcohpar$
   is an isomorphism.
   The map $\alpha \colon \cohpar \isom \Bunrel 1$ was defined using these isomorphisms
   and it therefore follows from
   \cref{eq:geometric-chapter-aut-e-1-via-proper-base-change}
   that
   $(\ave \bR i_! \tildeLaumon[1])|_{\Bunrel 1} = \Aut_E|_{\Bunrel 1}$.

   It remains to prove that
   $\ave \bR i_! \tildeLaumon$
   is zero outside of $\Bunrel 1$.
  \Cref{eq:geometric-chapter-aut-e-1-via-proper-base-change}
  tells us that $\Aut_E^1$ is supported on the image of $\cHcE \to \Bunpar^1$.
  The results of \cref{sec:relevant-locus}
  show that this images contains only two points outside of $\Bunrel 1$:
  $(\cO(1), \emptyset) \oplus (\cO, D)$
  and $(\cO(2), \emptyset) \oplus (\cO(-1), D)$.
  It therefore suffices to prove that the fiber of
  $\ave \bR i_! \tildeLaumon$ at each of these points is zero.

  For each $y \in \bP^1$,
  there are exactly two non-generic orbits in the modifications of $\pEtilde$ at $y$;
  the corresponding subbundles of $\pEtilde$ are
  $(\cO(1), \emptyset) \oplus (\cO, D)$
  and $(\cO(2), \emptyset) \oplus (\cO(-1), D)$.
  As a result,
  denoting by $\pE$ one of these two points,
  we find
  there is an isomorphism
  $\beta \colon q_{\pEtilde}^{-1}(\pE) \isom \BAut(\pEtilde) \times \bP^1$
  and a section $s \colon \bP^1 \to \barcohpar$
  of the support map,
  that fit into the commutative diagram
  \begin{equation*}
    \begin{tikzcd}
      q_{\pEtilde}^{-1}(\{\pE\}) \ar[rr, isom, swap, "\beta"]
      \ar[rd, swap, "p_{\pEtilde}"]
      &&
      \BAut(\pEtilde) \times \bP^1
      \arrow{ld}{\id \times s} \\
      & \BAut(\pEtilde) \times \barcohpar
    \end{tikzcd}
  \end{equation*}
  This shows
  \begin{equation*}
    (\ave \bR i_! \tildeLaumon)|_{\pE}
    = \rH_{\mathrm c}^\ast(\BAut(\pEtilde) \times \bP^1, (id \times s)^* \tildeLaumon)
  \end{equation*}
  and this is zero by
  \cref{prop:cohomology-groups-of-intermediate-extension-are-zero}.
\end{proof}

\begin{prop}
  \label{lem:extension-by-zero-is-normal-extension}
  The natural map
  $\ave \bR i_! \tildeLaumon
  \to \ave \bR i_* \tildeLaumon$
  is an isomorphism.
\end{prop}
\begin{proof}
  We denote by $\closurepEtilde \subset \Bunpar^1$
  the closure of $\pEtilde$.
  We prove that for every $F \in \boundedderivcat(\Bunpar^1, \bQ_\ell)$
  supported on $\closurepEtilde \setminus \{\pEtilde\}$,
  we have
  \begin{equation}
    \label{eq:claim-closure-doesnt-matter}
    \ave(F \boxtimes \tildeLaumon) = 0.
  \end{equation}
  The proposition then follows by applying the derived functor $\ave$
  to the distinguished triangle
 \begin{equation*}
   \bR i_! \tildeLaumon \to \bR i_* \tildeLaumon \to F \boxtimes \tildeLaumon \xrightarrow{+1}.
 \end{equation*}

 We can reduce the proof of \cref{eq:claim-closure-doesnt-matter} for all $F$ supported on
 $\closurepEtilde \setminus \{\pEtilde\}$
 to the case where $F$ is a skyscraper sheaf.
 Indeed,
 for every $\pE \in \Bunpar^1$,
 there are finitely many $\pE_i \in \Bunpar^2$ with
 $\pE \in q(p^{-1}(\pE_i))$;
 hence, to prove $\ave (F \boxtimes \tildeLaumon)|_{\pE} = 0$,
 we can construct a finite stratification of $q^{-1}(\pE)$
 by intersecting it with the fibers $p^{-1}(\{\pE_i\} \times \barcohpar)$.
 The compact cohomology of the stratum corresponding to $\pE_i$
 vanishes if $\ave(F_i \boxtimes \tildeLaumon)|_{\pE}$ is zero for every $F_i$
 supported on $\pE_i$ ---
 this completes the reduction step.
 %\note{This paragraph a bit cleaner with the Fourier transform interpretation...}
 %\note{This paragraph summarizes reduction 12.14 to 12.13 in my thesis.}

 \begin{lemma}
   Let $\pE \in \closurepEtilde \setminus \{\pEtilde\}$.
   Then $\pE$ and all its length 1 lower modifications are
   direct sums of parabolic line bundles.
 \end{lemma}
 \begin{proof}
   Since the boundary of $\{\pEtilde\}$ has dimension smaller than the dimension of $\{\pEtilde\}$,
   every $\pE \in \closurepEtilde \setminus \{\pEtilde\}$ has an automorphism group
   of dimension at least 3.
   Taking a lower modification of a bundle reduces the dimension of its automorphism group by at most one,
   so the proof of the lemma is concluded by remarking that
   all parabolic vector bundles $\pF$ in $\Bunpard 2$
   with an automorphism group of dimension at least 2
   are direct sums of line bundles.
   If the underlying bundle $\cF$ is $\cO(1) \oplus \cO(1)$,
   this follows from the fact that the automorphisms of $\cO(1) \oplus \cO(1)$
   act as M\"obius transformations on the flags
   and only scalar multiples of the identity fix three distinct points in $\bP^1$.
   If $\cF$ is $\cO(2 + n) \oplus \cO(-n)$ for some $n \in \bZ_{\geq 0}$,
   we can take an automorphism of the form
   $\begin{pmatrix}
     \lambda & \sigma \\ 0 & 1
   \end{pmatrix}$
   with $\lambda \in \bF_q^*$
   and $\sigma \colon \cO(-n) \to \cO(2 + n)$
   and then
   $\begin{pmatrix}
     1 & \frac 1 {\lambda - 1} \sigma \\ 0 & 1
   \end{pmatrix}$
   induces an isomorphism from $\pF$
   to a direct sum of parabolic line bundles.
   \end{proof}

 Let $F$ be a skyscraper sheaf supported on $\pE \in \closurepEtilde \setminus \{\pEtilde\}$.
 By the lemma,
 we can write $\pE = (\cL, I) \oplus (\cM, D \setminus I)$
 for line bundles on $\bP^1$
$\cL, \cM$
 and $I \subset D$.
 Every length one modification of $\pE$ is then of one of the following forms:
 \begin{align*}
   (\cL, I) &\oplus (\cM(-1), D \setminus I),  && \\
   (\cL(-1), I) &\oplus (\cM, D \setminus I), && \\
   (\cL(-1), I \cup \{x\}) &\oplus (\cM, D \setminus I \setminus \{x\}) &&  \text{for $x \in D \setminus I$},\\
   \text{or} \quad
   (\cL, I \setminus \{x\}) &\oplus (\cM(-1), D \setminus I \cup \{x\})  && \text{for $x \in I$}.
 \end{align*}

 Let $\pF$ denote one of these length 1 lower modifications.
 Then
 \begin{equation}
   \label{eq:proof-closure-doesnt-matter-express-fiber}
   (\ave (F \boxtimes \tildeLaumon))|_{\pF}
   = \rH^*_c(q^{-1}(\pF) \cap p^{-1}(\pE), p^* (F \boxtimes \tildeLaumon)).
 \end{equation}
 The map
 \begin{equation}
   \label{eq:proof-closure-doesnt-matter-map-from-fiber}
   q^{-1}(\pF) \cap p^{-1}(\pE) \to \barcohpar
 \end{equation}
 induced by $p \colon \cH \to \Bunpar^2 \times \barcohpar$
 is
 \begin{itemize}
 \item a section $\bP^1 \xrightarrow{s} \barcohpar$ of the support map
   if $\pF \cong (\cL, I) \oplus (\cM(-1), D \setminus I)$;
 \item
   the map $\cM^\vee \otimes \cL \to \bP^1 \xrightarrow{s} \barcohpar$,
   where $s \colon \bP^1 \to \barcohpar$ is again a section of the support map,
   if $\pF \cong (\cL(-1), I) \oplus (\cM, D \setminus I)$; or otherwise
 \item
   one of the two maps
   $k_x^{(\dash,0)}, k_x^{(0,\dash)} \colon \bA^1 \to \barcohpar$
   (\cref{rmk:cohomological-property-of-e-via-maps})
   for some $x \in D$.
 \end{itemize}

 Another way to describe the above,
 is to realize that the modifications of $\pF$
 are classified by a space isomorphic to $\bP(\cF) \cup_{D} \bP^1(\cF^{(-1,D)})$.
 The first map corresponds to a section of $\bP^1(\cF) \to \bP^1$;
 the second to the complement of this section;
 and the remaining 4 disjoint pieces correspond to the last maps.

It follows from
the cohomological properties of $E$ given in
\cref{lem:stalks-of-intermediate-extension-of-e}
and
\cref{rmk:cohomological-property-of-e-via-maps},
that the pullback of $\tildeLaumon$ along any of these maps
has vanishing compact cohomology.
Hence,
the expression in \cref{eq:proof-closure-doesnt-matter-express-fiber}
vanishes,
which completes the proof.
\end{proof}

\begin{cor}
  $\Aut_E^1 = \ave i_{!*} \tildeLaumon[1]$.
\end{cor}
\begin{proof}
  It follows from
  $\ave \bR i_! \tildeLaumon = \ave \bR i_{*} \tildeLaumon$
  (\cref{lem:extension-by-zero-is-normal-extension}) 
  that we have
  $\ave \bR i_! \tildeLaumon = \ave i_{!*} \tildeLaumon$.
  This is equal to $\Aut_E^1$ by 
  \cref{lem:aute1-via-fourier-transform}.
\end{proof}

\begin{proof}[Proof of \cref{thm:irreducibility-perversity-of-aute}]
  \label{proof-of-thm:irreducibility-perversity-of-aute}
  The irreducibility follows from the irreducibility of $E$
  and hence of $\cL_E$ and its pushforward along $\alpha$.
  To prove that $\Aut_E^1$ is perverse,
  we note that $p^* i_{!*} \tildeLaumon$ is perverse,
  because both $i_{!*}$ and $p^*$ preserve perversity (up to shift;
  $p$ is smooth);
  then by applying the decomposition theorem to $\bR q_!$,
  we find that
  $\bR q_! p^* i_{!*} \tildeLaumon[1] = \Aut_E^1$
  is also perverse.
\end{proof}

\section{Decomposition of the Hecke transform and compactification}
\label{sec:decomposition-of-hecke-transform-and-compactification}

\begin{prop}
  \label{prop:hecke-transform-decomposes}
  % 13.4 in thesis
  The complex $\heckeglobal \Aut_E^0$ decomposes
  as a direct sum of shifted simple perverse sheaves.
\end{prop}
\begin{proof}[Proof of
  \cref{prop:hecke-transform-decomposes},
  using \cref{sec:compactification} and \cref{sec:decomp-clean-extension-in-compactification}]
  In \cref{sec:compactification},
  we compactify the map $p \colon \cH \to \Bunpar^1 \times \barcohpar$
  (which fails to be proper over $\Bunpar^1 \times \Supp^{-1}(D)$)
  to a map $\bar p \colon \cHcomp \to \Bunpar^1 \times \barcohpar$
  (\cref{lem:barp-is-proper}),
  which gives us the following commutative diagram:
  \begin{equation}
    \label{eq:commutative-diagram-for-compactification-and-decomp}
    \begin{tikzcd}
      & \cHcomp \ar[ld, bend right, "\bar p"]
      & \\
      \Bunpar^1 \times \barcohpar
      & \cH \ar[l, "p"] \ar[r, swap,"q"] \ar[u, incl, "j"]
      & \Bunpar^0.
    \end{tikzcd}
  \end{equation}
  In \cref{sec:decomp-clean-extension-in-compactification},
  we show
  $\bR j_{!} q^* \Aut_E^0 = j_{!*} q^* \Aut_E^0$
  (\cref{prop:clean-extension-for-compactification}),
  from which it follows that
  \begin{equation}
    \heckeglobal \Aut_E^0
    := \bR p_! q^* \Aut_E^0
    = \bR \bar p_! j_{!*} q^* \Aut_E^0.
  \end{equation}
  Since $q$ is smooth,
  $j_{!*} q^* \Aut_E^0$
  is a shifted perverse sheaf,
  and we can apply the decomposition theorem to
  $\bR \bar p_!$
  to conclude the proof.
\end{proof}

\subsection{Compactification}
\label{sec:compactification}
Instead of $p \colon \cH \to \Bunpar^1 \times \barcohpar$,
we will compactify its dual
\begin{equation*}
  p^\vee \colon \cH = \cH^0 \to \Bunpar^{-1} \times \barcohpar,
  \qquad
  (\phi \colon \pE \surj \pT) \mapsto
  (\ker \phi, [\pT]).
\end{equation*}
We also introduce the dual of $q$, which is
\begin{equation*}
  q^\vee \colon \cH^0 \to \Bunpar^0, \qquad
  (\phi \colon \pE \surj \pT) \mapsto \pE.
\end{equation*}

These maps are dual in the following sense.
Dualization of vector bundles induces
a dualization of parabolic vector bundles
\begin{equation*}
  \mathrm{dual} \colon \Bunpar^d \isom \Bunpar^{-d},
  \qquad \pE \mapsto \dualvb{\pE}
\end{equation*}
and on length one coherent sheaves,
we define a duality
\begin{equation*}
  \mathrm{dual} \colon \cohpar \isom \cohpar,
  \qquad
  \pT \mapsto \iExt^1(\pT, \cO_{\bP^1}),
\end{equation*}
where $\iExt^1(\dash, \cO_{\bP^1})$ is applied degree-wise,
i.e., the degree $(i,x)$-part of $\iExt^1(\pT, \cO_{\bP^1})$
is $\iExt^1(\cT^{(-i,x)}, \cO_{\bP^1})$.
Note that there is a canonical isomorphism from
$\iExt^1(\cT^{(-i,x)}, \cO_{\bP^1})$ to
$\iHom(\cT^{(-i,x)}, \cO_{\Supp \cT}) \otimes \cO(\Supp \cT^{(-i,x)})$
(\cite[Lemma A.2]{DonagiPantev}).
% or: lemma 4.8 \label{lem:dual-of-torsion-sheaf-in-terms-of-hom}
These dualizing maps fit into the commutative diagram
\begin{equation*}
  \begin{tikzcd}[column sep=huge]
    \Bunpar^{- 1} \times \barcohpar \ar[d, isom, "\mathrm{dual}" swap]
    & \cH^0
    \arrow[swap]{l}{p^\vee}
    \arrow{r}{q^\vee}
    \ar[d, isom, "\mathrm{dual}" swap]
    & \Bunpar^{0} \ar[d, isom, "\mathrm{dual}" swap] \\
    \Bunpar^{1} \times \barcohpar
    & \cH  ^{1}
    \arrow{l}{p}
    \arrow[swap]{r}{q}
    & \Bunpar^{0}
  \end{tikzcd}
\end{equation*}
Hence, a compactification of $p^\vee$ provides us with a compactification of $p$.
The former is easier to describe.

We define an embedding $j \colon \cH \incl \cHcomp$ as follows.
We define a substack $\Bunextd 0$ of the moduli stack $\Coh_{2,D}^0$
of coherent parabolic sheaves of rank 2 and degree zero by
\begin{equation}
  \label{eq:definition-bunextd}
  \Bunextd 0 :=
  \stackbuilder{\pE \in \Coh_{2,D}^0}
  {\text{the torsion part $\pT$ of $\pE$} \\
    \text{is such that} \\
    \cT^{0} \oplus \cT^{(-1, D)} \\
    \text{has length 1}
  }.
\end{equation}
This substack contains $\Bunpar^0$.
The condition on the torsion part $\pT$
implies $\Supp \pT \subset D$.
We then define
\begin{equation}
  \label{eq:definition-cHcomp}
  \cHcomp :=
  \stackbuilder{0 \to \pF \to \pE \to \pT \to 0}
  {
    \pF \in \Bunpar^{-1}, \\
    \pE \in \Bunextd 0, \\
    \pT \in \cohpar
  },
\end{equation}
i.e., in $\cH$ we require the extension $\pE$ to lie in $\Bunpar^0$,
while here we only demand $\pE \in \Bunextd 0$.

\begin{lemma}
  \label{lem:barp-is-proper}
  The map
  \begin{equation*}
    \bar p \colon \cHcomp \to \Bunpar^{-1} \times \barcohpar,
    \qquad (\pF \incl \pE \surj \pT) \mapsto (\pF, [\pT])
  \end{equation*}
  is proper
  and $p^\vee = \bar p \circ j$.
\end{lemma}
\begin{proof}
    The moduli stack
  \begin{equation*}
    \underline{\Ext}^1(\pT_\univ, \pF_\univ) \to \Bunpar^{-1} \times \cohpar
  \end{equation*}
  is a vector bundle.
  The fibers of this map are $\bA^2$.
  The compactification $\cHcomp$ is the complement in $\underline{\Ext}^1(\pT_\univ, \pF_\univ)$ of the zero section
  \begin{equation*}
    \cHcomp =
    \underline{\Ext}^1(\pT_\univ, \pF_\univ) \setminus \Bunpar^{-1} \times \cohpar
  \end{equation*}
  and by mapping to $\barcohpar$ instead of $\cohpar$,
  we divide out the scaling,
  so that $\cHcomp$ is a projective bundle over $\Bunpar^{-1} \times \barcohpar$.
  The last statement is obvious from the definition.
\end{proof}

\subsection{Clean extension}
\label{sec:decomp-clean-extension-in-compactification}

In this section,
we prove the following proposition.
\begin{prop}
  \label{prop:clean-extension-for-compactification}
  Let $j \colon \cH \incl \cHcomp$ denote the inclusion.
  Then
  \begin{equation*}
    \bR j_{!} q^* \Aut_E^0 = j_{!*} q^* \Aut_E^0.
  \end{equation*}
\end{prop}

We start with the following lemma.
\begin{lemma}
  \label{lem:compactified-hecke-stack-smooth-over-bunplus}
  The map
  \begin{equation*}
    \bar q^\vee \colon
    \cHcomp \to \Bunextd 0,
    \qquad (\pF \incl \pE \surj \pT) \mapsto \pE.
  \end{equation*}
  is smooth and surjective and fits into the Cartesian square
  \begin{equation*}
    \begin{tikzcd}
      \cH \ar[r, incl, "j"] \ar[d, "q^\vee"]
      & \cHcomp \ar[d, "\bar q^\vee"] \\
      \Bunpar^0 \ar[r, incl, "j_+"]
      & \Bunextd 0
    \end{tikzcd}
  \end{equation*}
\end{lemma}
\begin{proof}
  The surjectivity follows from the fact that
  for any $\pE \in \Bunextd 0$,
  an inverse image is given by
 $(T_x \pE \incl \pE \surj (\pE/(T_x \pE))|_x) \in \cHcomp$.
 The square is Cartesian by definition of the stacks $\cHcomp$ and $\Bunextd 0$.

 To prove that the map is smooth,
 we show that for every
 $(\pF \incl \pE \surj \pT) \in \cHcomp$,
 the map on tangent spaces
 \begin{equation*}
   \Ext^1(\pF \incl \pE, \pF \incl \pE) \to \Ext^1(\pE, \pE)
 \end{equation*}
 is surjective.
 Let $\xi \in \Ext^1(\pE, \pE)$;
 we will construct an inverse image.
 By functoriality of $\Ext^1(\pE, \dash)$,
 we can use the map $\pE \surj \pT$ to construct an extension
 $\xi_0 \in \Ext^1(\pE, \pT)$.
 Every such extension comes from an extension
 $\xi_1 \in \Ext^1(\pT, \pT)$ by pullback along $\pE \surj \pT$:
 indeed,
 the short exact sequence $0 \to \pF \to \pE \to \pT \to 0$
 induces a long exact sequence of which a part is
 \begin{equation*}
   \Ext^1(\pT, \pT) \to \Ext^1(\pE, \pT) \to \Ext^1(\pF, \pT)
 \end{equation*}
 and $\Ext^1(\pF, \pT)$ is zero because $\pF$ is torsion-free by definition of $\cHcomp$.
 This gives us a surjective map of extensions $\xi \surj \xi_1$,
 whose kernel is a in $\Ext^1(\pF, \pF)$;
 the resulting short exact sequence of extensions is the sought-after preimage.
\end{proof}

Because of this lemma,
we can prove \cref{prop:clean-extension-for-compactification}
by showing $j_{+, !*} \Aut_E^0 = \bR j_{+,!} \Aut_E^0$.
This follows from the following lemma.

\begin{lemma}
  % Proposition 13.12 in thesis
  Let the notation be as above.
  Then the natural map
  \begin{equation*}
    \bR j_{+,!} \Aut_E^0
    \to
    \bR j_{+,*} \Aut_E^0
  \end{equation*}
  is an isomorphism.
\end{lemma}
\begin{proof}
  For every $x \in D$,
  we denote by $\Bunextdeven x 0 \subset \Bunextd 0$
  the substack classifying $\pE \in \Bunextd 0$ whose torsion part lives in even degree and is supported
  at $x$.
  (The proof is analogous for parabolic sheaves with torsion part in odd degree.)
  Let $j_x \colon \Bunpar^0 \incl \Bunextdeven x 0$ denote the inclusion.
  We prove the lemma by proving that the map
  $\bR j_{x,!} \Aut_E^0 \to \bR j_{x,*} \Aut_E^0$
  is an isomorphism for all $x \in D$.

  Let $\cHthreeplus$ denote the stack
\begin{equation*}
  \cHthreeplus
  :=
  \stackbuilder
  {
    0 \to \pF \to \pE \to k_x \to 0
  }
  {
    \pF \in \Bun_{2,D \setminus \{x\}}^{-1}, \\
    \pE \in \Coh_{2,D \setminus \{x\}}^0, \\
    k_x \in \Coh_0^1 \; \text{length 1, supported on $x$}
  };
\end{equation*}
i.e., it classifies extensions of $k_x$ by parabolic sheaves
$\cF \in \Bun_{2,D \setminus \{x\}}^{-1}$
with parabolic structure only at $D \setminus \{x\}$.
There is an isomorphism
\begin{equation*}
  \pi_x \colon \cHthreeplus \isom \Bunextdeven x 0, \qquad
  (\pF \incl \pE \surj k_x) \mapsto (\pE, \cF \incl \cE)
\end{equation*}
where $(\pE, \cF \incl \cE)$ denotes the parabolic sheaf in $\Coh_{2,D}^0$
whose parabolic structures at $D \setminus \{x\}$ are given by $\pE$,
and whose parabolic structure at $x$ is
\begin{equation*}
  \mathellipsis \incl \cE(-x) \incl \cF \incl \cE \incl \cF(x) \incl \mathellipsis
\end{equation*}

Let $\pE \in \Bunextd 0 \setminus \Bunpar^0$
and write
\begin{equation*}
  \pi_x^{-1}(\pE) =:
  (0 \to \cE^{\bullet, (-1,x)} \incl \cE^{\bullet, (0,x)} \surj k_x \to 0) \in \cHthreeplus.
\end{equation*}
Then we claim
\begin{equation}
  \label{eq:calculation-of-fiber-to-prove-clean-extension}
  (\bR j_{x,*} \Aut_E^0)|_{\pE}
  = \rH^*( \Ext^1(k_x, \cE^{\bullet,(-1,x)}) \setminus \{0\}, \; \pi_x^* \Aut_E^0).
\end{equation}
\begin{proof}[Proof of \eqref{eq:calculation-of-fiber-to-prove-clean-extension}]
  The map
  \begin{equation*}
    \begin{split}
      f \colon \cHthreeplus &\to \Bun_{2,D \setminus \{x\}}^{-1} \times \BAut(k_x),
      \\
      (\pF_{D \setminus \{x\}} \incl \pE_{D \setminus \{x\}} \surj k_x) &\mapsto (\pF_{D \setminus \{x\}}, [k_x]).
    \end{split}
  \end{equation*}
  is a vector bundle
  (\cite[Remark 6.3 (1)]{heinloth2004})
  and its zero section
  \begin{equation*}
    s \colon
    \Bun_{2,D \setminus \{x\}}^{-1} \times \BAut(k_x)
    \to \cHthreeplus
  \end{equation*}
  is exactly the complement of $\im (\pi_x^{-1} \circ j_x)$.
  We write
  $\autelift := \pi_x^*(\Aut_E^0)$
  for the pullback to $\pi_x^{-1}(\Bunpar^0)$.
  A general lemma
  (\cite[Lemma 0.3]{heinloth2004} is the statement we use;
  the calculation appears in \cite{FGVConjecture} and \cite{Brylinski86})
  then says that because $\autelift$ is $\Gm$-equivariant
  for the natural $\Gm$-action on the vector bundle $\cHthreeplus \to \Bun_{2, D \setminus \{x\}}^{-1} \times \BAut(k_x)$,
  we have
  \begin{equation*}
    s^* \left(  \bR j_{x,*}\autelift \right) = \bR f_* \autelift.
  \end{equation*}
  The statement then follows by pulling back $\bR f_* \autelift$
  along $$(\cE^{\bullet, (-1,x)}, k_x) \colon \Spec k \to \Bun_{2,D \setminus \{x\}}^{-1} \times \BAut(k_x).$$
\end{proof}

It remains to show that
$\rH^*( \Ext^1(k_x, \cE^{\bullet,(-1,x)}) \setminus \{0\}, \; \pi_x^* \Aut_E^0)$
is zero.
The map
\begin{equation*}
\Ext^1(k_x, \cE^{\bullet, (-1, x)}) \setminus \{0\} \to \cHthreeplus \xrightarrow{\pi_x} \Bunpar^0
\end{equation*}
factors as
\begin{equation*}
  \Ext^1(k_x, \cE^{\bullet, (-1, x)}) \setminus \{0\}
  \xrightarrow{\rho} \bP(\cF|_x)
  \xrightarrow{\phi_{\pF}} \Bunpar^{-1}
  \xrightarrow{T_x^{-1}} \Bunpar^0
\end{equation*}
where $\rho$ maps $(i \colon \pF \incl \pE)$
to $\ker(i|_x \colon \cF|_x \to \cE|_x)$
and $\phi_{\pF}$ maps
$\ell \subset \cF|_x$
to $(\pF, \ell) \in \Bunpar^{-1}$.
By the projection formula,
the cohomology group we are calculation
(\cref{eq:calculation-of-fiber-to-prove-clean-extension})
is equal to
\begin{equation*}
  \rH^*(\bP^1(\cF|_x), \phi_{\pF}^* T_{x,*} \Aut_E^0 \otimes \bR \rho_* \bQ_\ell).
\end{equation*}
Because $\bR \rho_* \bQ_\ell$
is an extension of two constant local systems,
this group is zero if
\begin{equation*}
  \rH^*(\bP^1(\cF|_x), \phi_{\pF}^* T_{x,*} \Aut_E^0) = 0
\end{equation*}
holds.
This last equality follows from the cohomological properties of $E$ listed in
\cref{sec:cohomological-properties-of-e}.
See \cite[Lemma 13.14]{mythesis} for details.
\end{proof}

\section{The Hecke transform is an intermediate extension}
\label{sec:hecke-transform-is-intermediate-extension}

The goal of the this section is to prove the following theorem.
\begin{thm}
  \label{thm:interemediate-extension-of-local-system}
  The complex $\heckeglobal \Aut_E^0$
  on $\Bunpar^1 \times \barcohpar$
  \begin{enumerate}[(a)]
  \item is the intermediate extension of a rank 4 local system on
    $\pi^{-1}(\bP^1 \setminus D) \times \bP^1 \setminus D$; and
  \item
    vanishes outside of $\Bunrel 1 \times \barcohpar$.
  \end{enumerate}
\end{thm}
We conclude the proof of this theorem on page
\pageref{pf:proof-of-thm-interemediate-extension-of-local-system}.

\begin{prop}
  \label{prop:aute-0-is-local-system-on-open}
  The restriction of the complex $\heckeglobal \Aut_E^0$ on $\Bunpar^1 \times \barcohpar$
  to $\pi^{-1}(\bP^1 \setminus D) \times (\bP^1 \setminus D)$
  is a local system of rank 4.
\end{prop}
\begin{proof}
  Since $\heckeglobal \Aut_E^0$ decomposes as a direct sum of shifted perverse sheaves
  (\cref{prop:hecke-transform-decomposes}),
  it suffices to prove that all fibers have dimension 4.
  By definition,
  for any point $(\pE, [\pT]) \in \pi^{-1}(\bP^1 \setminus D) \times (\bP^1 \setminus D)$,
  we have
  \begin{equation*}
    \begin{split}
      (\heckeglobal \Aut_E^0[-2])|_{(\pE, [\pT])}
      &= (\bR p_! q^* \Aut_E^0)|_{(\pE, [\pT])} \\
      &= \rH_{\mathrm c}^*( p^{-1}((\pE, [\pT])), \; \phi^* \Aut_E^0)
    \end{split}
  \end{equation*}
  where
  \begin{equation*}
    \phi \colon p^{-1}((\pE, [\pT])) \to \Bunrel 0
  \end{equation*}
  is the restriction of $q \colon \cH \to \Bunpar^0$.
  This map $\phi$
  has degree 2
  ($p^{-1}((\pE, [\pT]))$
  is isomorphic to $\bP^1$)
  and its image does indeed lie in $\Bunrel 0$
  (\cref{thm:summary-statement-of-modifications-normal-points}).

  We claim that $\phi^* \Aut_E^0$ is irreducible.
  Indeed, assume towards a contradiction
  that there exists a rank one local system $L$ on $p^{-1}((\pE, [\pT]))$
  and a non-zero map $\phi^* \Aut_E^0 \to L$.
  By adjunction, this map corresponds to a non-zero map
  $\Aut_E^0 \to \phi_* L$,
  which is an isomorphism,
  since both local systems are rank 2 and $\Aut_E^0$ is irreducible.
  This contradicts the assumption that $\Aut_E^0|_{\Bunrel 0} = \phi_* L|_{\Bunrel 0}$ has unipotent monodromy,
  because the monodromy of every square in $\pi_1(\Bunrel 0)$ acts semisimply on $\phi_* L$.

  As a result, $\phi^* \Aut_E^0$ does not have any global section,
  and by duality,
  $\rH_{\mathrm c}^2(p^{-1}((\pE, [\pT])), \phi^* \Aut_E^0)$ vanishes, too.
  It therefore suffices to prove that the Euler characteristic of $\phi^* \Aut_E^0$ is -4.
  This is a straightforward computation using Grothendieck-Ogg-Shafarevich
  (\cite[formula 7.2]{sga5}, or \cite[theorem  9.1]{kindlerruelling2014});
  we sketch it here.
  Denote by $S \subset \bP^1 = p^{-1}((\pE, [\pT]))$
  the ramification locus, which consists of two points.
  Write $D' = \phi^{-1}(\pi_0^{-1}(D))$.
  Then $\#D' = 8 - s$,
  where $s = \#(D' \cap S)$.
  By the additivity of the Euler characteristic, we have
  \begin{equation*}
    \chi(\bP^1, \phi^* \Aut_E^0)
    = \chi_c(\bP^1 \setminus D', \phi^* E)
    + \chi(D', \phi^* \Aut_E^0),
  \end{equation*}
  where $\chi_c$ denotes the Euler characteristic with compact support.
  By Grothendieck-Ogg-Shafarevich,
  \begin{equation*}
    \chi_c(\bP^1 \setminus D', \phi^* E)
    = \rank (\phi^* E)  \cdot\chi_c(\bP^1 \setminus D', \bQ_\ell) = 2 \cdot(2 - \#D') = -12 + 2s.
  \end{equation*}

  We calculate $\chi(D', \phi^* \Aut_E^0)$
  using our determination of the stalks of $\Aut_E^0$
  (\cref{lem:stalks-of-intermediate-extension-of-e});
  the essential point is that for $x \in D' \cap S$,
  $\phi(x)$ has $\Gm \times \Gm$ automorphisms
  (and $\chi(\Spec k, (\Aut_E^0)|_{\phi(x)}) = 0$),
  while for $x \in D' \setminus S$,
  $\phi(x)$ has $\Gm$-automorphisms
  (and $\chi(\Spec k, (\Aut_E^0)|_{\phi(x)}) = 1$).
  Hence $\chi(D', \phi^* \Aut_E^0) = 8 - 2s$,
  which completes the proof.
\end{proof}

\begin{prop}
  \label{prop:aut-e-0-vanishes-outside-relevant-locus}
  The complex $\heckeglobal \Aut_E^0$ vanishes outside of $\Bunrel 1 \times \barcohpar$.
\end{prop}
\begin{proof}
  Let $\pE \in \Bunpar^1 \setminus \Bunrel 1$.
  We distinguish two cases:
  $\dim \Aut(\pE) \geq 3$
  and $\dim \Aut(\pE) = 2$.
  (All parabolic bundles $\pF \in \Bunpar^1$ with $\dim \Aut(\pF) = 1$ lie in the relevant locus.)

  Suppose $\dim \Aut(\pE) \geq 3$.
  Then no length 1 lower modification of $\pE$ lies in $\Bunrel 0$.
  Indeed,
  if $\pF \subset \pE$ were a length 1 lower modification of $\pE$ with $\pF \in \Bunrel 0$,
  then $\pE(-1) \subset \pF$ would be a length lower modification of $\pF \in \Bunrel 0$
  with $\dim \Aut(\pE(-1)) \geq 3$,
  in contradiction to our calculations of the length 1 lower modifications of all parabolic bundles
  in the relevant locus
  (see \cref{sec:calculation-of-all-length-1-lower-modifications}).
  Because $\Aut_E^0$ is supported on $\Bunrel 0$,
  this proves that
  $(\heckeglobal \Aut_E^0)|_{\{\pE\} \times \barcohpar} = 0$.

  Suppose now that $\dim \Aut(\pE) = 2$.
  There are exactly eight parabolic bundles in $\Bunpar^0 \setminus \Bunrel 0$
  with an automorphism group of dimension 2, namely:
  $\pE_\irrel := (\cO(2), \emptyset) \oplus (\cO(-1), D)$,
  $T_D \pE_\irrel(2)$
  and $T_{x_1} T_{x_2} \pE_\irrel(1)$ for every $x_1, x_2 \in D$ with $x_1 \neq x_2$.
  With a straightforward calculation,
  one can show that
  for every $[\pT] \in \barcohpar$,
  the map
  \begin{equation*}
    \phi_{(\pE_\irrel, [\pT])} \colon p^{-1}((\pE_\irrel, [\pT])) \cap q^{-1}(\Bunrel 0) \to \Bunrel 0
  \end{equation*}
  that is the restriction of $q$,
  is either
  $k_x^{(\dash, 0)} \colon \bA^1 \to \barcohpar$
  or
  $k_x^{(0, \dash)} \colon \bA^1 \to \barcohpar$
  (\cref{rmk:cohomological-property-of-e-via-maps})
  for some $x \in D$.
  It follows
  that the fiber
  \begin{equation*}
    (\heckeglobal \Aut_E^0)|_{(\pE, [\pT])}
    = \rH_c^{\ast}( p^{-1}(\pE_\irrel, [\pT]) \cap q^{-1}(\Bunrel 0), \; \phi_{(\pE_\irrel, \pT)}^* \Aut_E^0 )[2]
  \end{equation*}
  vanishes (\cref{rmk:cohomological-property-of-e-via-maps}).
  Because all relevant operations commute with the elementary Hecke operators $T_{x'}$ ($x' \in D$)
  and with twisting by $\cO(1)$ or $\cO(2)$,
  the same holds for the seven other points in $\Bunpar^0 \setminus \Bunrel 0$
  with 2-dimensional automorphism groups.
\end{proof}

In previous sections,
we have drawn conclusions on
the stalks of $\heckeglobal \Aut_E^0$
at points in $\pi^{-1}(\bP^1 \setminus D) \times (\bP^1 \setminus D)$.
We have also seen that 
$\heckeglobal \Aut_E^0$
vanishes outside of $\Bunrel 1 \times \barcohpar$.
In this section,
we consider the stalks at the remaining points:
the points $(\pE, [\pT]) \in \Bunrel 1 \times \barcohpar$
such that $\pi(\pE)$ or $\Supp \pT$ lies in $D$.
The fibers at these points are easier to calculate than the other fibers in $\Bunrel 1 \times \barcohpar$:
in this case,
the restriction of $q \colon \cH \to \Bunpar^0$
to $p^{-1}((\pE, [\pT])) \cap q^{-1}(\Bunrel 0)$
is a degree 1 map
when composed with $\pi_0 \colon \Bunrel 0 \to \bP^1$,
whereas it is degree 2 in the other cases.
The fiber of $\heckeglobal \Aut_E^0$ at $(\pE, [\pT])$
is the compact cohomology of the pullback of $\Aut_E^0[2]$
along these maps,
and these cohomology groups can therefore be deduced directly from
our classification of these maps in 
\cref{sec:calculation-of-all-length-1-lower-modifications}.
This gives the following result.
Recall that for $x \in D$, $M_x$ denotes the unique M\"obius transformation
$\bP^1 \isom \bP^1$
that preserves $D$ and sends $\infty$ to $x$.

\begin{prop}
  \label{prop:stalks-of-hecke-transform-of-aute0-when-at-least-one-point-is-special}
  Let $(\pE, \pT) \in \Bunrel 1 \times \cohpar$.
  Write
  \begin{equation*}
    \{x, y\} = \{ \pi(\pE), \; \Supp(\pT)\} \subset \bP^1
  \end{equation*}
  and assume $x \in D$,
  i.e., we have $\pE \in \pi_0^{-1}(D)$,
  $\pT \in \Supp^{-1}(D)$ or both.
  % \begin{align*}
  %   x &:= \Supp(\alpha^{-1}(\pE)) \qquad \text{and} \\
  %   y &:= \Supp(\pT).
  % \end{align*}
  \begin{enumerate}
  \item
    If $\pE$ and $\pT$ both have automorphism group $\Gm$,
    then
    \begin{equation*}
      (\heckeglobal \Aut_E^0)|_{(\pE, [\pT])}
      = (j_{!*} E[1])|_{M_x(y)} \otimes \Ecoeff.
    \end{equation*}
  \item
    If either $\pE$ or $\pT$, but not both, has $\Gm \times \Gm$ as its automorphism group,
    then
    \begin{equation*}
      (\heckeglobal \Aut_E^0)|_{(\pE, [\pT])}
      = (j_{!*} E[2])|_{M_x(y)} \otimes \rH_c^*(\Gm, \bQ_\ell) \otimes \Ecoeff.
    \end{equation*}
  \item
    If both $\pE$ and $\pT$ have automorphism group $\Gm \times \Gm$,
    then
    \begin{equation*}
      (\heckeglobal \Aut_E^0)|_{(\pE, [\pT])}
      = (j_{!*} E[2])|_{M_x(y)} \otimes \rH_c^*(\Gm, \bQ_\ell) \times \rH^*(\Gm, \bQ_\ell) \otimes \Ecoeff.
    \end{equation*}
  \end{enumerate}
\end{prop}
\begin{proof}
  This follows directly from our classification of these maps in 
  \cref{sec:calculation-of-all-length-1-lower-modifications};
  see the explanation just above the proposition.
\end{proof}

\begin{cor}
  \label{cor:cohomological-degrees-of-hecketransform-of-aute0}
  The cohomology sheaves of the complex $(\heckeglobal \Aut_E^0)|_{\Bunrel 1 \times \barcohpar}$ 
  satisfy the following:
  \begin{align*}
    \sheafcohom^i(\heckeglobal \Aut_E^0) &= 0 \qquad \text{for $i < -1$ and $i > 1$,} \\
    \sheafcohom^{-1}(\heckeglobal \Aut_E^0) &\text{ is supported on an open of $\Bunrel 1 \times \barcohpar$,} \\
    \sheafcohom^{0}(\heckeglobal \Aut_E^0) &\text{ is supported on an codimension 2 set, and} \\
    \sheafcohom^{1}(\heckeglobal \Aut_E^0) &\text{ is supported on an codimension 4 set.}
  \end{align*}
\end{cor}
\begin{proof}
  This is immediate from the proposition and the fact that a point
  $\pE \in \Bunpar^1$ lies in codimension $\dim \Aut(\pE)$.
\end{proof}

\begin{proof}[Proof of \cref{thm:interemediate-extension-of-local-system}]
  \label{pf:proof-of-thm-interemediate-extension-of-local-system}
  Let $i \colon Z \incl \Bunpar^1 \times \barcohpar$ be a locally closed embedding,
  $L$ a local system on $Z$ and $r \in \bZ$ such that
  $i_{!*} L[r]$
  is one of the shifted simple perverse sheaves in the decomposition of $\heckeglobal \Aut_E^0$.
  The cohomology sheaves of $i_{!*} L[r]$ satisfy the conditions in
  \cref{cor:cohomological-degrees-of-hecketransform-of-aute0};
  in particular,
  \begin{equation}
    \label{eq:proof-of-thm-interemediate-extension-of-local-system-dimension-counting}
    \begin{split}
      \cH^i(L[r]) &= 0 \text{ for all $i \in \bZ$ with $i < s_0$ or $s_1 < i$, where} \\
      (s_0, s_1) &=
      \begin{cases}
        (-1, -1) & \text{if $\dim Z = 1, 0$} \\
        (-1, 0) & \text{if $\dim Z = -1, -2$} \\
        (-1, 1) & \text{if $\dim Z \leq -3$}
      \end{cases}.
    \end{split}
  \end{equation}

  Let $m \in \bZ$ denote the perverse degree of $i_{!*} L[r]$.
  Then $r = \dim Z - m$
  and by \cref{eq:proof-of-thm-interemediate-extension-of-local-system-dimension-counting},
  \begin{equation*}
    s_0 \leq \dim Z - m \leq s_1.
  \end{equation*}

  The Verdier dual of $i_{!*} L[r]$
  is a simple perverse sheaf in the decomposition of
  $\pdual \heckeglobal \Aut_E^0 = \heckeglobal \Aut_{E^\vee}^0$,
  and hence satisfies the same conditions.
  It lies in perverse degree $-m$, so we conclude that
  $s_0 \leq \dim Z + m \leq s_1$ also holds.
  Adding these two inequalities,
  we find $s_0 \leq \dim Z \leq s_1$,
  and this can only hold for $\dim Z = 1 = \dim(\Bunpar^1 \times \barcohpar)$.
  Hence, every simple perverse sheaf in the decomposition of $\heckeglobal \Aut_E^0$
  comes from a local system on a dense open.
  Since $(\heckeglobal \Aut_E^0)|_{\pi^{-1}(\bP^1 \setminus D) \times (\bP^1 \setminus D)}$
  is a local system
  (\cref{prop:aute-0-is-local-system-on-open}),
  we conclude that $\heckeglobal \Aut_E^0$ is indeed the intermediate extension of its
  restriction to $\pi^{-1}(\bP^1 \setminus D) \times (\bP^1 \setminus D)$.
  We had already proven that $\heckeglobal \Aut_E^0$ vanishes outside of $\Bunrel 1 \times \barcohpar$
  (\cref{prop:aut-e-0-vanishes-outside-relevant-locus}),
  so this completes the proof.
\end{proof}

\section{Proof of the Hecke property}
\label{sec:proof-of-hecke-property}

The goal of this section is
\cref{thm:aute-is-indeed-hecke-eigensheaf},
which says that $\Aut_E$ is the Hecke eigensheaf associated to the local system $E$.

\begin{lemma}
  \label{lem:pullback-by-alpha-is-hecke-transform-restricted-to-e-tilde}
  % This is more or less lemma 13.31 in my thesis.
  % Really nothing happens in the proof, except for writing out the definitions.
  Let $F \in \boundedderivcat(\Bunpar^1, \bQ_\ell)$ be a complex that is supported on
  $\Bunrel 1 \subset \Bunpar^1$.
  Then $(\heckeglobal F)|_{\BAut(\pEtilde) \times \barcohpar}$
  descends to
  $\BAut(\pEtilde)/\Gm \times \barcohpar$.
  This descended complex is
  $\alpha^*(F[2]|_{\Bunrel 1})$.
\end{lemma}
\begin{proof}
  This follows quite easily from the definition of $\heckeglobal = \bR p_! q^*[2]$
  and the definition of $\alpha$,
  which was defined as $q^\rel \circ (p^\rel)^{-1}$,
  where $p^\rel$ and $q^\rel$
  are restrictions of $p$ and $q$,
  respectively,
  to a substack of $\cH$ that only classifies modifications of $\pEtilde$
  that lie in the relevant locus.
\end{proof}

We define the ``large diagonal''
\begin{equation*}
  \Delta^+ := \im (\barcohpar \times_{\bP^1} \barcohpar \incl \barcohpar \times \barcohpar)
\end{equation*}
where the fiber product over $\bP^1$ is with respect to the support map
$\Supp \colon \barcohpar \to \bP^1$.
By a symmetric complex of constructible sheaves on $\barcohpar \times \barcohpar$,
we mean a complex $F$ such that there exists an isomorphism
$\sigma^*(F) \isom F$,
where $\sigma$ denotes the automorphism
of $\barcohpar \times \barcohpar$
that interchanges the two factors in the product.

\begin{prop}
  \label{prop:hecke-transform-aute0-is-symmetric}
  The complex $(\alpha \times \id)^* ((\heckeglobal \Aut_E^0)|_{\Bunrel 1 \times \barcohpar})$
  on $\cohpar \times \barcohpar$
  descends to a complex on $\barcohpar \times \barcohpar$
  whose restriction to $(\barcohpar \times \barcohpar) \setminus \Delta^+$ is symmetric.
\end{prop}
\begin{proof}
  By \cref{lem:pullback-by-alpha-is-hecke-transform-restricted-to-e-tilde},
  $(\alpha \times \id)^*((\heckeglobal \Aut_E^0)|_{\Bunrel 1 \times \barcohpar})$
  can be identified with a shift of
  \begin{equation*}
    ((\heckeglobal \times \id_{\boundedderivcat(\barcohpar)}) \heckeglobal \Aut_E^0)|_{\BAut(\pE) \times \barcohpar \times \barcohpar},
  \end{equation*}
  i.e., where we apply the global Hecke operator twice,
  which means we push and pull along the composition of correspondences

  \begin{minipage}{\textwidth} % de minipage en de twee hidewidths zijn om het goed te centreren
    \begin{equation*}
      \hidewidth
      \Bunpar^2 \times \barcohpar \times \barcohpar
      \xleftarrow{p \times \id} \cH \times \barcohpar \xrightarrow{q \times \id}
      \Bunpar^1 \times \barcohpar
      \xleftarrow{p} \cH \xrightarrow{q}
      \Bunpar^0.
      \hidewidth
    \end{equation*}
  \end{minipage}

  Because modifications at different points in $\bP^1$ commute,
  the restriction of this complex to $\Bunpar^2 \times ((\barcohpar \times \barcohpar) \setminus \Delta^+)$
  is symmetric with respect to the automorphism that interchanges the two factors $\barcohpar$.
\end{proof}

\begin{lemma}
  \label{lem:hecke-transform-of-aute0-has-one-correct-fiber}
  Let $x \in D$
  and let $\pT$ be either
  $k_{x}^{(1,0)}$ or $k_x^{(0,1)}$.
  Then
  \begin{equation*}
    (\heckeglobal \Aut_E^0)|_{\Bunrel 1 \times \{\pT\}}
    \cong \alpha_*(j_{!*} M_x^*(E)[1]) \otimes \Ecoeff.
  \end{equation*}
  Likewise,
  \begin{equation*}
    (\heckeglobal \Aut_E^0)|_{\alpha(\pT)  \times \barcohpar}
    \cong \alpha_*(j_{!*} M_x^*(E)[1]) \otimes \Ecoeff.
  \end{equation*}

\end{lemma}
\begin{proof}
  This follows from
  our calculations of all the stalks $(\heckeglobal \Aut_E^0)|_{(\pE, \pT)}$
  when either $\pE$ or $\pT$ lies over $D$
  (\cref{prop:stalks-of-hecke-transform-of-aute0-when-at-least-one-point-is-special}).
\end{proof}

\begin{prop}
  \label{prop:aut-e-is-eigensheaf-in-degree-0}
  Let $E$ be an irreducible pure rank 2 local system on $\bP^1 \setminus D$.
  Then there is an isomorphism
  \begin{equation*}
    \heckeglobal \Aut_E^0 \cong \Aut_E^1 \boxtimes j_{!*} E[1].
  \end{equation*}
\end{prop}
\begin{proof}
  We denote
  by $\heckeglobal_\alpha \Aut_E^0$
  the complex on $\barcohpar \times \barcohpar$
  descended from the pullback
  $(\alpha \times \id)^* ((\heckeglobal \Aut_E^0)|_{\Bunrel 1 \times \barcohpar})$
  (as in \cref{prop:hecke-transform-aute0-is-symmetric}).
  Let $j \colon \bP^1 \setminus D \incl \barcohpar$ denote the inclusion.
  Because $\heckeglobal \Aut_E^0$ is the intermediate extension of a rank 4 local system on
  $\pi^{-1}(\bP^1 \setminus D) \times (\bP^1 \setminus D)$
  (\cref{thm:interemediate-extension-of-local-system})
  and
  because $E$ was assumed to be pure,
  there exist
  local systems $F_i, G_i$ on $\bP^1 \setminus D$
  such that
  \begin{equation}
    \label{eq:direct-summand-decomposition-in-proof-of-main-theorem}
    \heckeglobal_\alpha \Aut_E^0 =
    (j \times j)_{!*}
    \left(
      \bigoplus_{i} F_i \boxtimes G_i
    \right).
  \end{equation}
  Since the restriction of $\heckeglobal_\alpha \Aut_E^0$ to
  $(\barcohpar \times \barcohpar) \setminus \Delta^+$
  is symmetric,
  we can choose $F_i, G_i$ such that if $F_i$ is not isomorphic to $G_i$,
  then $G_i \boxtimes F_i$ is also one of the summands in the decomposition
  (\cref{eq:direct-summand-decomposition-in-proof-of-main-theorem}).

  Because the restriction of $\heckeglobal_\alpha \Aut_E^0$
  to $\barcohpar \times k_x^{(1,0)}$
  (for any $x \in D$)
  is isomorphic to (a M\"obius pullback of) $j_{!*} E$
  (up to shift and scalar; \cref{lem:hecke-transform-of-aute0-has-one-correct-fiber}),
  which is the intermediate extension of an irreducible local system,
  there is at least one direct summand $F \boxtimes G$
  in the decomposition of \ref{eq:direct-summand-decomposition-in-proof-of-main-theorem}
  of rank at least 2.
  If we can prove it has rank 4,
  then it is the only direct summand,
  and is therefore of the form $F \boxtimes F$.
  
  Suppose that $F \boxtimes G$ has rank 3.
  Then $F$ and $G$ are not isomorphic,
  so the rank 3 summand $G \boxtimes F$ also appears in
  the direct sum
  \cref{eq:direct-summand-decomposition-in-proof-of-main-theorem},
  but this leads to the contradiction
  that the restriction of $\heckeglobal_\alpha \Aut_E^0$
  to $(\bP^1 \setminus D) \times (\bP^1 \setminus D)$,
  which is a local system of rank 4,
  contains a rank 6 local system  $(F \boxtimes G) \oplus (G \boxtimes F)$.

  Suppose that $F \boxtimes G$ has rank 2.
  Then again $F$ and $G$ are not isomorphic,
  so $G \boxtimes F$ is one of the other summands
  in the decomposition,
  and in fact the only other summand.
  But this would imply that
  $(\heckeglobal_\alpha \Aut_E^0)|_{(k_x^{(1,0)}, k_x^{(1,0)})}$
  has rank 2,
  in contradiction 
  to \cref{lem:hecke-transform-of-aute0-has-one-correct-fiber}
  and the fact that $(j_{!*} E)|_{k_x^{(1,0)}}$ has rank 1 for all $x \in D$
  (\cref{lem:stalks-of-intermediate-extension-of-e}).

  We conclude that there exists an irreducible local system $F$ of rank 2 such that
  $\heckeglobal_\alpha \Aut_E^0$ is isomorphic to
  $(j \times j)_{!*}(F \boxtimes F)$.
  The restriction to
  $\barcohpar \times \{k_\infty^{(1,0)}\}$
  is therefore
  $j_{!*} F \otimes (j_{!*} F)_{k_\infty^{(1,0)}}$;
  but by our calculations
  (\cref{lem:hecke-transform-of-aute0-has-one-correct-fiber}),
  this is also equal to $j_{!*} E \otimes (j_{!*} E)|_{k_\infty^{(1,0)}}$.
  We therefore conclude $F = E$,
  which completes the proof.
 \end{proof}

 \begin{thm}
   \label{thm:aute-is-indeed-hecke-eigensheaf}
   Let $E$ be an irreducible pure rank 2 local system on $\bP^1 \setminus D$.
   Then there is an isomorphism
   \begin{equation*}
     \heckeglobal \Aut_E \isom \Aut_E \boxtimes j_{!*}E[1].
   \end{equation*}
 \end{thm}
 \begin{proof}
   Consider the commutative diagram
   \begin{equation*}
     \begin{tikzcd}
      \Bunpar^2 \times \barcohpar \arrow{d}{(T_\infty, T_\infty)}
      & \cH^2 \ar[l, "p_2", swap] \ar[r, "q_2"] \arrow{d}{T_\infty}
      & \Bunpar^1 \arrow{d}{T_\infty} \\
      \Bunpar^1 \times \barcohpar
      & \cH^1 \ar[l, "p_1", swap] \ar[r, "q_1"]
      & \Bunpar^0
    \end{tikzcd}
  \end{equation*}
  Pulling back and pushing forward $\Aut_E^0$ along the various maps gives us
  \begin{equation*}
    \begin{tikzcd}
      T_{\infty}^* (\Aut_E^1) \boxtimes T_\infty^* (j_{!*} E)[-1]
      & \Aut_E^1 \otimes E|_\infty
      \\
      \Aut_E^1 \boxtimes j_{!*} E[-1] \arrow[mapsto]{u}{(T_\infty, T_\infty)^*}
      & \Aut_E^0 \arrow[mapsto]{l}{\bR p_{1,!} q_1^*}  \arrow[mapsto]{u}{T_\infty^*}
    \end{tikzcd}
  \end{equation*}
  where $\bR p_{1,!} q_1^* \Aut_E^0 = \heckeglobal[-2] \Aut_E^0 = \Aut_E^1 \boxtimes j_{!*} E[-1]$
  by \cref{prop:aut-e-is-eigensheaf-in-degree-0}.
  Because $T_\infty \colon \barcohpar \to \barcohpar$ is the identity on $\bP^1 \setminus D$,
  we have $T_\infty^* (j_{!*} E) = j_{!*} E$
  and therefore
  \begin{equation*}
      T_{\infty}^* (\Aut_E^1) \boxtimes T_\infty^* (j_{!*} E)[-1]
      = (\Aut_E^2 \otimes \Ecoeff) \boxtimes j_{!*} E[-1]
  \end{equation*}
  The commutativity of the diagram hence implies
  \begin{equation*}
    \heckeglobal \Aut_E^1
    = \bR p_{2,!} q_{2}^* \Aut_E^0[2]
    = \Aut_E^2 \boxtimes j_{!*} E.
  \end{equation*}
  We can repeat this argument in the other degrees to conclude.
\end{proof}

% \printindex

% \nocite{*} % dit is geloof ik als je alles in je bibliography wil hebben wat niet geciteerd is
\bibliographystyle{alpha}
\bibliography{general-bibliography}

\end{document}